\documentclass{amsart}
\usepackage{amsmath,amsthm,amssymb}
\usepackage{hyperref}
\usepackage{pstricks,pst-node,multido}

\usepackage{a4wide,graphicx}

\title{Generalized Dyck tilings}
\author{Matthieu Josuat-Verg\`es and Jang Soo Kim}
\thanks{The first author was partially supported by the ANR project
  CARMA.  The second author was supported by Basic Science Research
  Program through the National Research Foundation of Korea (NRF)
  funded by the Ministry of Education (NRF-2013R1A1A2061006).}

\newtheorem{thm}{Theorem}[section]
\newtheorem{lem}[thm]{Lemma}
\newtheorem{prop}[thm]{Proposition}
\newtheorem{cor}[thm]{Corollary}
\theoremstyle{definition}

\newtheorem{defn}[thm]{Definition}

\newtheorem{problem}{Problem}
\theoremstyle{remark}
\newtheorem{remark}{Remark}
\newcommand\D{\mathcal{D}}
\newcommand\kH{\mathcal{H}^{(k)}}
\newcommand\kD{\D^{(k)}}
\newcommand\kS{\mathfrak{S}^{(k)}}
\newcommand\Sym{\mathfrak{S}}
\newcommand\tiles{\mathrm{tiles}}

\newcommand\INC{\mathrm{INC}}
\newcommand\MARK{\mathrm{MARK}}
\newcommand\post{\mathrm{post}}
\newcommand\pre{\mathrm{pre}}
\newcommand\sym{\mathrm{sym}}
\newcommand\UP{\mathrm{UP}}

\newcommand\NC{\mathrm{NC}}
\newcommand\Nest{\mathrm{Nest}}

\newcommand\M{\mathcal{M}}
\newcommand\SP{\mathfrak{S}^{(k)}}
\newcommand\HH{\mathcal{H}}

\newcommand\LL{\mathcal{L}}

\newcommand\area{\mathrm{area}}

\newcommand\cro{\mathrm{cr}}
\newcommand\sscr{\mathrm{sscr}}

\newcommand\wt{\mathrm{wt}}

\newcommand\kDyck{\mathrm{Dyck}^{(k)}}
\def\sch.{Schr{\"o}der}

\newcommand\HT{\mathrm{ht}}
\newcommand\lm{\lambda/\mu}

\newcommand\qint[1]{\left[ #1\right]_q}

\newcommand\Qbinom[3]{\genfrac{[}{]}{0pt}{}{#1}{#2}_{#3}}
\newcommand\qbinom[2]{\Qbinom{#1}{#2}{q}}

\newcommand\inv{\mathrm{inv}}
\newcommand\INV{\mathrm{INV}}
\newcommand\art{\mathrm{art}}
\newcommand\pos{\mathrm{pos}}
\newcommand\norm[1]{\left\lVert#1\right\rVert}

\psset{gridlabels=0pt,subgriddiv=1, griddots=4, gridwidth=.1pt,gridcolor=gray}
\psset{linewidth=1.5pt, unit=.4cm}
\psset{unit=9pt,dimen=middle}

\def\cross(#1,#2){\rput(#1,#2){$\times$}}

\def\kdyckgrid#1#2{
\psgrid(0,0)(#1,#2)
\psline[linewidth=.1pt, linecolor=gray](0,0)(#1,#2)
}

\def\cvput#1[#2]{\pnode(#1,1){#1} \pscircle*(#1,1){.1} \rput(#1,.5){$#2$}}

\begin{document}

\begin{abstract}
Recently, Kenyon and Wilson introduced Dyck tilings, which are certain
tilings of the region between two Dyck paths.  The enumeration of Dyck
tilings is related with hook formulas for forests and the
combinatorics of Hermite polynomials.  The first goal of this work is
to give an alternative point of view on Dyck tilings by making use of
the weak order and the Bruhat order on permutations.  Then we
introduce two natural generalizations: $k$-Dyck tilings and symmetric
Dyck tilings.  We are led to consider Stirling permutations, and
define an analog of the Bruhat order on them.  We show that certain
families of $k$-Dyck tilings are in bijection with intervals in this
order.  We also enumerate symmetric Dyck tilings.
\end{abstract}

\maketitle

% \tableofcontents

\section{Introduction}

Dyck tilings were recently introduced by Kenyon and Wilson
\cite{Kenyon2011} in the study of probabilities of statistical physics
model called ``double-dimer model'', and independently by Shigechi and
Zinn-Justin \cite{Shigechi2012} in the study of Kazhdan-Lusztig
polynomials. Dyck tilings also have connection with fully packed loop
configurations \cite{Fischer2012} and representations of the symmetric
group \cite{Fayers2013}.

The main purpose of this paper is to give a new point of view on Dyck
tilings in terms of the weak order and the Bruhat order on
permutations and to consider two natural generalizations of Dyck
tilings.

A \emph{Dyck path of length $2n$} is a lattice path consisting of up
steps $(0,1)$ and down steps $(1,0)$ from the origin $(0,0)$ to the
point $(n,n)$ which never goes strictly below the line $y=x$. We will also
consider a Dyck path $\lambda$ of length $2n$ as the Young diagram
whose boundary is determined by $\lambda$ and the lines $x=0$ and $y=n$.

Suppose that $\lambda$ and $\mu$ are Dyck paths of length $2n$ with
$\mu$ weakly above $\lambda$.  A \emph{Dyck tile} is a ribbon such that
the center of the cells form a Dyck path. A \emph{Dyck tiling} of $\lm$ is a
tiling $D$ of the region between $\lambda$ and $\mu$ with Dyck tiles
satisfying the \emph{cover-inclusive property}: if $\eta$ is a tile of
$D$, then the translation of $\eta$ by $(1,-1)$ is either completely
below $\lambda$ or contained in another tile of $D$. See
Figure~\ref{fig:Dyck_tiling} for an example. We denote by $\D(\lm)$
the set of Dyck tilings of $\lm$.  For $D\in\D(\lm)$ we call $\lambda$
and $\mu$ the \emph{lower path} and the \emph{upper path} of $D$,
respectively.  Then the set of Dyck tilings with fixed upper path
$\lambda$ is denoted by $\D(\lambda/*)$ and similarly, the set of Dyck
tilings with fixed lower path $\mu$ is denoted by $\D(*/\mu)$. 

For $D\in \D(\lm)$ we have two natural statistics: the area $\area(D)$
of the region $\lm$ and the number $\tiles(D)$ of tiles of $D$. We
also consider the statistic $\art(D)=(\area(D)+\tiles(D))/2$.

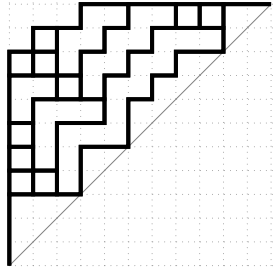
\begin{figure}
  \centering
  \begin{pspicture}(0,0)(11,11) \kdyckgrid{11}{11} \psline(0,0)(0,
    1)(0, 2)(0, 3)(0, 4)(0, 5)(0, 6)(0, 7)(0, 8)(0, 9)(1, 9)(1, 10)(2,
    10)(3, 10)(3, 11)(4, 11)(5, 11)(6, 11)(7, 11)(8, 11)(9, 11)(10,
    11)(11, 11)\psline(0,0)(0, 1)(0, 2)(0, 3)(1, 3)(2, 3)(3, 3)(3,
    4)(3, 5)(4, 5)(5, 5)(5, 6)(5, 7)(6, 7)(6, 8)(7, 8)(7, 9)(8, 9)(9,
    9)(9, 10)(9, 11)(10, 11)(11, 11)\pspolygon(3, 3)(2, 3)(2, 4)(2,
    5)(2, 6)(3, 6)(4, 6)(4, 7)(4, 8)(5, 8)(5, 9)(6, 9)(6, 10)(7,
    10)(8, 10)(9, 10)(9, 9)(8, 9)(7, 9)(7, 8)(6, 8)(6, 7)(5, 7)(5,
    6)(5, 5)(4, 5)(3, 5)(3, 4)(3, 3)\pspolygon(9, 10)(8, 10)(8, 11)(9,
    11)\pspolygon(8, 10)(7, 10)(7, 11)(8, 11)\pspolygon(2, 3)(1, 3)(1,
    4)(2, 4)\pspolygon(2, 4)(1, 4)(1, 5)(1, 6)(1, 7)(2, 7)(3, 7)(4,
    7)(4, 6)(3, 6)(2, 6)(2, 5)(2, 4)\pspolygon(4, 7)(3, 7)(3, 8)(3,
    9)(4, 9)(4, 10)(5, 10)(5, 11)(6, 11)(7, 11)(7, 10)(6, 10)(6, 9)(5,
    9)(5, 8)(4, 8)(4, 7)\pspolygon(3, 7)(2, 7)(2, 8)(3,
    8)\pspolygon(3, 8)(2, 8)(2, 9)(2, 10)(3, 10)(3, 11)(4, 11)(5,
    11)(5, 10)(4, 10)(4, 9)(3, 9)(3, 8)\pspolygon(1, 3)(0, 3)(0, 4)(1,
    4)\pspolygon(1, 4)(0, 4)(0, 5)(1, 5)\pspolygon(1, 5)(0, 5)(0,
    6)(1, 6)\pspolygon(1, 6)(0, 6)(0, 7)(0, 8)(1, 8)(2, 8)(2, 7)(1,
    7)(1, 6)\pspolygon(2, 8)(1, 8)(1, 9)(2, 9)\pspolygon(2, 9)(1,
    9)(1, 10)(2, 10)\pspolygon(1, 8)(0, 8)(0, 9)(1, 9)\end{pspicture}
  \caption{An example of Dyck tiling.}
  \label{fig:Dyck_tiling}
\end{figure}

Kenyon and Wilson \cite{Kenyon2011} conjectured the following two
formulas:
\begin{equation}
  \label{eq:KW_lambda}
\sum_{D\in \D(\lambda/*)} q^{\art(D)} = 
\frac{[n]_q!}{\prod_{x\in F} [h_x]_q},
\end{equation}
\begin{equation}
  \label{eq:KW_mu}
\sum_{D\in \D(*/\mu)} q^{\tiles(D)} = \prod_{u\in\UP(\mu)} [\HT(u)]_q,
\end{equation}
where $F$ is the plane forest corresponding to $\lambda$ and, for a
vertex $x\in F$, $h_x$ denotes the hook length of $x$.  The set of up
steps of a Dyck path $\mu$ is denoted by $\UP(\mu)$ and for
$u\in\UP(\mu)$, $\HT(u)$ is the number of squares between $u$ and
the line $y=x$ plus 1.  Here we use the standard notation for
$q$-integers and $q$-factorials: $[n]_q=1+q+q^2+\dots+q^{n-1}$ and $[n]_q! = [1]_q
[2]_q \dots[n]_q$.

Formula~\eqref{eq:KW_lambda} was first proved by Kim \cite{JSK_DT}
non-bijectively and then by Kim, M\'esz\'aros, Panova, and Wilson
\cite{KMPW} bijectively. In \cite{KMPW}, they find a bijection between
$\D(\lambda/*)$ and increasing labelings of the plane forest
corresponding to $\lambda$. Kim~\cite{JSK_DT} and Konvalinka
independently proved \eqref{eq:KW_mu} by finding a bijection between
$\D(*/\mu)$ and certain labelings of $\mu$ called Hermite histories.

Bj\"orner and Wachs showed that the right hand side of
\eqref{eq:KW_lambda} is the length generating function for
permutations in an interval in the weak order, see
\cite[Theorem~6.8]{Bjorner1991} and \cite[Theorem~6.1]{Bjorner1989}.

In this paper we first show that, using the results of Bj\"orner and
Wachs \cite{Bjorner1989,Bjorner1991}, \eqref{eq:KW_lambda} can be
interpreted as the length generating function for permutations
$\pi\ge_L\sigma$ in the left weak order for a 312-avoiding permutation
$\sigma$. We also show that \eqref{eq:KW_mu} is the length generating
function for permutations $\pi\ge \sigma$ in the Bruhat order for a
132-avoiding permutation $\sigma$.  We then consider two natural
generalizations of Dyck tilings, namely, $k$-Dyck tilings and
symmetric Dyck tilings.

The first generalization is $k$-Dyck tiling, where we use $k$-Dyck
paths and $k$-Dyck tiles with the same cover-inclusive property.  We
generalize \eqref{eq:KW_mu} by finding a bijection between $k$-Dyck
tilings and $k$-Hermite histories. We consider $k$-Stirling
permutations introduced by Gessel and Stanley \cite{Gessel1978}. We
define a $k$-Bruhat order on $k$-Stirling permutations and show that
$k$-Dyck tilings with fixed upper path are in bijection with an
interval in this order. We also consider a connection with $k$-regular
noncrossing partitions.  We generalize \eqref{eq:KW_lambda} to
$k$-Dyck tilings with fixed lower path $\lambda$ when $\lambda$ is a
zigzag path.

The second generalization is symmetric Dyck tiling, which is invariant
under the reflection along a line. We show that symmetric Dyck tilings
are in bijection with symmetric matchings and ``marked'' increasing
labelings of symmetric forests. 

\section{Dyck tilings as intervals of the Bruhat order and weak order}

As we have seen in the introduction, the two natural points of view
for enumerating Dyck tilings are when we fix the upper path, and when
we fix the lower path. We show in this section that both can be
interpreted in terms of permutations, using respectively the Bruhat
order and the (left) weak order, see \cite{BjornerBrenti}.

We begin with the case of a fixed upper path, and the Bruhat order.

We denote by $\Sym_n$ the set of permutations of
$[n]:=\{1,2,\dots,n\}$. An \emph{inversion} of $\pi\in\Sym_n$ is a
pair $(i,j)$ of integers $1\le i < j\le n$ such that
$\pi(i)>\pi(j)$. The number of inversions of $\pi$ is denoted by
$\inv(\pi)$.  For a permutation $\tau$, the set of $\tau$-avoiding
permutations in $\Sym_n$ is denoted by $\Sym_n(\tau)$.
For example if $\tau =132$, $\sigma\in \Sym_n(132)$ if there is no 
$i<j<k$ such that $\sigma_i<\sigma_k<\sigma_j$.

We represent a permutation by a diagram with the ``matrix convention'', i.e. there is
a dot at the intersection of the $i$th line from the top and the $j$th column from the left
if $\sigma(j)=i$. In these diagrams, we can represent the inversion of a permutation
by putting a cross $\times$ in each cell having a dot to its right and a dot below.
See the left part of Figure~\ref{bijcatalan}.
We need a bijection $\alpha$ between 132-avoiding permutations and Dyck paths. It
is easy to see that the inversions of a 132-avoiding permutation are top left justified
in its diagram. So we can define a path from the bottom left corner to the top right corner
by following the boundary of the region filled with $\times$. This turns out to be a Dyck path and 
this defines a bijection (this is an easy exercise).
The bijection is illustrated in Figure~\ref{bijcatalan}.

\begin{defn}
Let $\mu$ be a Dyck path. 
\begin{itemize}
\item
A {\it Hermite history} of shape $\mu$ is a 
labelling of the up steps of $\mu$ with integers such that 
a step starting at height $h$ has a label between $0$ and $h$.
\item
A {\it matching} of shape $\mu$ is a partition of $[n]$ in 2-element blocks
such that $i\in[n]$ is the minimum of a block if and only if the $i$th step 
of $\mu$ is an up step. A {\it crossing} of the matching is a pair of blocks
$\{i,j\}$ and $\{k,\ell\}$ such that $i<j<k<\ell$.
\end{itemize}
\end{defn}

The following is well known (see for example \cite{JSK_DT}).

\begin{prop}
There is a bijection between Hermite histories of shape $\mu$ and matchings of shape $\mu$.
It is such that the sum of weights in the Hermite history is the number of crossings in the matching. 
\end{prop}

\begin{figure} \centering
\begin{pspicture}(5,5)
  \psgrid[gridcolor=gray,griddots=4,subgriddiv=0,gridlabels=0](0,0)(5,5)
  \psdots(0.5,2.5)(1.5,1.5)(2.5,3.5)(3.5,4.5)(4.5,0.5)
  \rput(0.5,4.5){ \textcolor{gray}{$\times$}  }
  \rput(1.5,4.5){ \textcolor{gray}{$\times$}  }
  \rput(2.5,4.5){ \textcolor{gray}{$\times$}  }
  \rput(0.5,3.5){ \textcolor{gray}{$\times$}  }
  \rput(1.5,3.5){ \textcolor{gray}{$\times$}  }
\end{pspicture}
\qquad $\mapsto$ \qquad
\begin{pspicture}(5,5)
  \psgrid[gridcolor=gray,griddots=4,subgriddiv=0,gridlabels=0](0,0)(5,5)
  \psline(0,0)(0,3)(2,3)(2,4)(3,4)(3,5)(5,5)
  \rput(0.5,4.5){ \textcolor{gray}{$\times$}  }
  \rput(1.5,4.5){ \textcolor{gray}{$\times$}  }
  \rput(2.5,4.5){ \textcolor{gray}{$\times$}  }
  \rput(0.5,3.5){ \textcolor{gray}{$\times$}  }
  \rput(1.5,3.5){ \textcolor{gray}{$\times$}  }
\end{pspicture}
\caption{The bijection
from 132-avoiding permutations to Dyck paths.
The crosses represent inversions of the permutation $34215$. 
\label{bijcatalan} }
\end{figure}
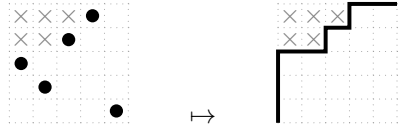

\begin{thm} \label{fixedupperpathbruhat}
Let $\sigma\in\Sym_n(132)$ and $\mu=\alpha(\sigma)$, then
\[
\sum_{D\in \D(*/\mu)} q^{\tiles(D)} = 
\sum_{\pi\ge\sigma} q^{\inv(\pi)-\inv(\sigma)},
\]
where $\pi\ge\sigma$ is the Bruhat order on $\Sym_n$.
\end{thm}
\begin{proof}
From \cite{KMPW}, we know that Dyck tilings with a fixed upper path $\mu$ are in bijection with Hermite
histories with shape $\mu$, and the bijection sends the number of tiles to the sum of labels in the
Hermite history.
Consequently, Dyck tilings with a fixed upper path $\mu$ are in bijection with matchings of shape $\mu$,
and the bijection sends the number of tiles to the number of crossings in the matching.

To show the proposition, we give a bijection between matchings with the same shape $\mu$, and 
permutations above $\sigma$ in the Bruhat order. It is illustrated in Figure~\ref{bijbruhat}.
The idea is to put dots in the grid as follows: if there is a pair $(i,j)$ in the matching
(with $i<j$), the $i$th step in the Dyck path is vertical and the $j$th step is horizontal, 
so row to the left of the $i$th step intersects the column below the $j$th step in some cell, and we put 
a dot in this cell. Then we can read these dots as a permutation (with the matrix convention). In the example
in Figure~\ref{bijbruhat} we get 45321. The crossing in the matchings correspond two inversions of the permutations that lay
below the Dyck path.

The next step is the following: we can prove the set of permutations where all dots are below the Dyck path
$\mu$ is precisely the Bruhat interval $\{ \pi \,:\, \pi \geq \sigma\}$.
First, by construction all the dots of $\sigma$ are below $\mu$. Suppose all the dots of $\pi$ are below $\mu$
and $\pi' \gtrdot \pi$ in the Bruhat order. It means that $\pi'$ is obtained from $\pi$ by transforming
a pair of dots arranged as 
\psset{unit=2mm}
\;\begin{pspicture}(1,1) \psdot(1,0)\psdot(0,1)  \end{pspicture}\;
into a pair of dots arranged as 
\;\begin{pspicture}(1,1) \psdot(1,1)\psdot(0,0)  \end{pspicture}\;, and the new dots cannot be above $\mu$.
So the interval $\{ \pi \,:\, \pi \geq \sigma\}$ is included in the set of permutations where all dots are 
below the Dyck path $\mu$. Reciprocally, let $\pi$ be a permutation where all dots are 
below the Dyck path $\mu$. If $\pi\neq\sigma$, consider an inversion of $\pi$ wich is as low to the right 
as possible. This inversion is in a pattern 
\;\begin{pspicture}(1,1) \rput(0,1){$\times$}\psdot(1,1)\psdot(0,0)  \end{pspicture}\;
and the cross is below $\mu$. By transforming this pattern into 
\;\begin{pspicture}(1,1) \psdot(1,0)\psdot(0,1)  \end{pspicture}\;, we obtain $\pi'$ whith $\pi'\lessdot\pi$
and has still the property that all dots are below $\mu$. By repeating this operation we must arrive at 
a permutation whose inversions are exactly the cells above $\mu$, i.e. $\sigma$. So $\pi$ is in the 
interval $\{ \pi \,:\, \pi \geq \sigma\}$.
\end{proof}

\begin{figure} \centering
 \begin{pspicture}(10,2)
   \psdots(1,0)(2,0)(3,0)(4,0)(5,0)(6,0)(7,0)(8,0)(9,0)(10,0)
   \psarc(3,0){2}{0}{180}\psarc(2.5,0){0.5}{0}{180}
   \psarc(5,0){1}{0}{180}\psarc(8,0){1}{0}{180}\psarc(9,0){1}{0}{180}
 \end{pspicture}
\qquad $\mapsto$ \qquad
\begin{pspicture}(5,5)
  \psgrid[gridcolor=gray,griddots=4,subgriddiv=0,gridlabels=0](0,0)(5,5)
  \psline(0,0)(0,2)(1,2)(1,3)(3,3)(3,5)(5,5)
  \psdots(0.5,1.5)
  \psdots(1.5,0.5)
  \psdots(2.5,2.5)
  \psdots(3.5,3.5)
  \psdots(4.5,4.5)
  \rput(1.5,2.5){ \textcolor{gray}{$\times$}  }
  \rput(3.5,4.5){ \textcolor{gray}{$\times$}  }
 \end{pspicture}
\caption{The bijection proving Proposition~\ref{fixedupperpathbruhat}. \label{bijbruhat} } 
\end{figure}
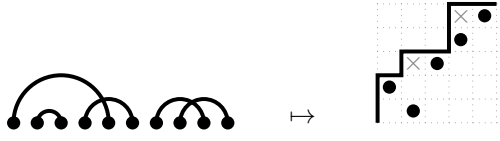

\begin{figure} \centering
\begin{pspicture}(14,3)
  \psgrid[gridcolor=gray,griddots=4,subgriddiv=0,gridlabels=0](0,0)(14,3)
  \psline(0,0)(2,2)(3,1)(4,2)(5,1)(7,3)(10,0)(12,2)(14,0)
\end{pspicture}
\qquad $\mapsto$ \qquad
\begin{pspicture}(-1,0)(3,2)
\psdots(0,0)(-1,1)(0,1)(1,1)(1,2)(3,0)(3,1)
\psline(0,0)(-1,1)
\psline(0,0)(0,1)
\psline(0,0)(1,1)
\psline(1,1)(1,2)
\psline (3,0)(3,1) 
\end{pspicture}
\caption{The bijection from Dyck paths to plane forests. Here the Dyck
  path is rotated clockwise by an angle of $45^\circ$.
\label{bijcatalan2} }
\end{figure}
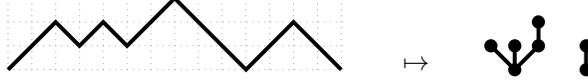

Let us turn to the case of a fixed lower path in Dyck tilings.  We can
identify a Dyck path $\lambda$ with a plane forest: it is obtained from $\lambda$ by ``squeezing'' 
the path as shown in Figure~\ref{bijcatalan2},
a pair of facing up step and down step corresponds to a vertex.

Let $F$ be a plane forest with $n$ vertices. An \emph{increasing
  labeling} of $F$ is a way of labeling the vertices of $F$ with
$1,2,\dots,n$ so that the label of a vertex is greater than the label
of its parent. Let $L$ be an increasing labeling of $F$. An
\emph{inversion} of $L$ is a pair $i<j$ such that $j$ is not a
descendant of $i$ and $j$ appears to the left of $i$. For example, if
$L$ is the increasing labeling in Figure~\ref{fig:post}, then $L$ has
many inversions including $(13,7),(8,7),(2,4)$.  We denote by
$\post(L)$ (respectively, $\pre(L)$) the permutation obtained by
reading $L$ from left to right using post-order (respectively,
pre-order). See Figure~\ref{fig:post}.  It is easy to see that there
is a unique increasing labeling $L_0$ of $F$ such that
$\inv(\post(L_0))$ is minimal. More specifically, $L_0$ is the
increasing labeling of $F$ with $\inv(L_0)=0$, or equivalently, $L_0$
is the increasing labeling of $F$ such that $\pre(L_0)=id_n$, the
identity permutation of $[n]$. It is not difficult to see that the
permutation $\pi_0=\post(L_0)$ is 312-avoiding.

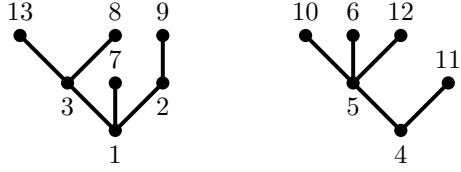
\begin{figure} \centering
\psset{unit=2}
\begin{pspicture}(1,-.5)(10,2.5)
\psdots(1,2)(2,1)(3,2)(3,0)(3,1)(4,1)(4,2)(7,2)(8,1)(8,2)(9,0)(9,2)(10,1)
\psline(1,2)(2,1)
\psline (3,2)(2,1)
\psline (3,0)(2,1)
\psline (3,0)(3,1)
\psline (3,0)(4,1)
\psline (4,1)(4,2)
\psline(7,2)(8,1)
\psline(8,1)(8,2)
\psline(8,1)(9,2)
\psline(8,1)(9,0)
\psline(9,0)(10,1)
\rput(1,2.5){13}
\rput(2,0.5){3}
\rput(3,2.5){8}
\rput(3,-0.5){1}
\rput(3,1.5){7}
\rput(4,.5){2}
\rput(4,2.5){9}
\rput(7,2.5){10}
\rput(8,.5){5}
\rput(8,2.5){6}
\rput(9,-0.5){4}
\rput(9,2.5){12}
\rput(10,1.5){11}
\end{pspicture}
\caption{An increasing labeling $L$ of a forest of size 13.
We have $\post(L)=13,8,3,7,9,2,1,10,6,12,5,11,4$ and
$\pre(L)=1,3,13,8,7,2,9,4,5,10,6,12,11$.}
\label{fig:post}
\end{figure}

\begin{thm}
  Let $\lambda$ be a Dyck path with corresponding plane forest
  $F$. Let $L_0$ be the increasing labeling of $F$ such
  that $\pre(L_0)=id_n$, and $\pi_0=\post(L_0)$.  Then
\[
\sum_{D\in \D(\lambda/*)} q^{\art(D)} = 
\sum_{\pi\ge_L\pi_0} q^{\inv(\pi)-\inv(\pi_0)},
\]
where $\ge_L$ is the left weak order on $\Sym_n$.
\end{thm}

\begin{proof}
  It is shown in \cite{KMPW} that
\[
\sum_{D\in \D(\lambda/*)} q^{\art(D)} = 
\sum_{L\in\LL(F)} q^{\inv(\pi)},
\]
where $\LL(F)$ is the set of increasing labelings of $F$. Thus, it is
enough to show that for all $k\ge0$, 
there is a bijection between the two sets 
\[
A_k=\{L\in\LL(F): \inv(L)=k\}, \quad
B_k=\{\pi\in\Sym_n: \pi\ge_L \pi_0, \inv(\pi)-\inv(\pi_0)=k\}.
\]

We will show that the map $L\mapsto \post(L)$ is a bijection from
$A_k$ to $B_k$ for all $k\ge0$ by induction on $k$. Since
$A_0=\{L_0\}$ and $B_0=\{\pi_0\}$, it is true when $k=0$. Suppose that
the claimed statement is true for $k\ge0$. We need to show that the
map $L\mapsto \post(L)$ is a bijection from $A_{k+1}$ to
$B_{k+1}$. Let $L\in A_{k+1}$. Since $\inv(L)=k+1\ge1$, we can find an
integer $i$ such that $(i+1,i)$ is an inversion of $L$. Since $i+1$ is
not a descendant of $i$ in $L$, the labeling $L'$ obtained from $L$ by
exchanging $i$ and $i+1$ is also an increasing labeling of $F$. Since
$\inv(L')=\inv(L)-1=k$ we have $L'\in A_k$. By the induction
hypothesis, $\pi' = \post(L') \in B_k$. One can easily see that the
permutation $\pi=\post(L)$ is obtained from $\pi'$ by exchanging $i$
and $i+1$. Since $i+1$ appears to the left of $i$ in $\pi$ we have
$\inv(\pi)=\inv(\pi')+1$ and $\pi\in B_{k+1}$. Thus $L\mapsto
\post(L)$ is a map from $A_{k+1}$ to $B_{k+1}$. Similarly, we can show
that, for given $\pi\in B_{k+1}$, there is $L\in B_{k+1}$ such that
$\post(L)=\pi$. Since $L$ is determined by $\post(L)$ for all $L\in
\LL(F)$, the map $L\mapsto \post(L)$ is a bijection from $A_{k+1}$ to
$B_{k+1}$. 
\end{proof}

Note that the inversion generating function of increasing labelings of a plane forest is given by a hook length 
formula \cite{Bjorner1989}.
The fact that some intervals for the weak order have a generating function given by a hook length formula
follows from \cite{Bjorner1991}.

\section{\texorpdfstring{$k$-Dyck tilings}{k-Dycktilings}}

For an integer $k\geq 1$,
a \emph{$k$-Dyck path} is a lattice path consisting of \emph{up steps}
$(0,1)$ and \emph{down steps} $(1,0)$ from the origin $(0,0)$ to the
point $(kn,n)$ which never goes below the line $y=x/k$.  Let
$\kDyck(n)$ denote the set of $k$-Dyck paths from $(0,0)$ to
$(kn,n)$. It is well known that the cardinal of $\kDyck(n)$ is 
the Fuss-Catalan number $\frac{1}{kn+1}\binom{(k+1)n}{n}$
(see for example \cite{dvoretzky}). As in the case of Dyck path,
we denote $\UP(\mu)$ the set of up steps of a $k$-Dyck path $\mu$.

A \emph{$k$-Dyck tile} is a ribbon in which the centers of the cells
form a $k$-Dyck path. 
Let $\lambda,\mu\in\kDyck(n)$ such that $\mu$ is weakly above $\lambda$.
A \emph{(cover-inclusive) $k$-Dyck tiling} is a
tiling $D$ of the region $\lm$ between $\lambda,\mu\in \kDyck(n)$ with
$k$-Dyck tiles satisfying the cover-inclusive property: if $\eta$ is
a tile in $D$, then the translation of $\eta$ by $(1,-1)$ is either
completely below $\lambda$ or contained in another tile of $D$. See
Figure~\ref{fig:2-DT} for an example.  We denote by $\kD(\lm)$ the set
of $k$-Dyck tilings of $\lm$. We also denote by $\kD(\lambda/*)$ and
$\kD(*/\mu)$ the sets of $k$-Dyck tilings with fixed lower path
$\lambda$ and with fixed upper path $\mu$, respectively.

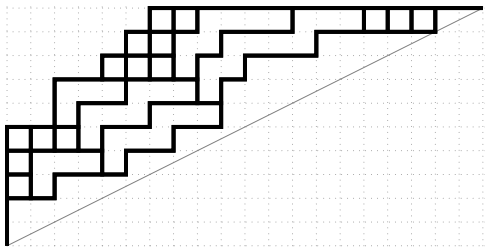
\begin{figure}
  \centering
\begin{pspicture}(0,0)(20,10) \kdyckgrid{20}{10}
 \psline(0,0)(0, 1)(0, 2)(0, 3)(0, 4)(0,
    5)(1, 5)(2, 5)(2, 6)(2, 7)(3, 7)(4, 7)(4, 8)(5, 8)(5, 9)(6, 9)(6,
    10)(7, 10)(8, 10)(9, 10)(10, 10)(11, 10)(12, 10)(13, 10)(14,
    10)(15, 10)(16, 10)(17, 10)(18, 10)(19, 10)(20, 10)\psline(0,0)(0,
    1)(0, 2)(1, 2)(2, 2)(2, 3)(3, 3)(4, 3)(5, 3)(5, 4)(6, 4)(7, 4)(7,
    5)(8, 5)(9, 5)(9, 6)(9, 7)(10, 7)(10, 8)(11, 8)(12, 8)(13, 8)(13,
    9)(14, 9)(15, 9)(16, 9)(17, 9)(18, 9)(18, 10)(19, 10)(20,
    10)\pspolygon(18, 9)(17, 9)(17, 10)(18, 10)\pspolygon(17, 9)(16,
    9)(16, 10)(17, 10)\pspolygon(5, 3)(4, 3)(4, 4)(4, 5)(5, 5)(6,
    5)(6, 6)(7, 6)(8, 6)(9, 6)(9, 5)(8, 5)(7, 5)(7, 4)(6, 4)(5, 4)(5,
    3)\pspolygon(16, 9)(15, 9)(15, 10)(16, 10)\pspolygon(2, 2)(1,
    2)(1, 3)(1, 4)(2, 4)(3, 4)(4, 4)(4, 3)(3, 3)(2, 3)(2,
    2)\pspolygon(9, 6)(8, 6)(8, 7)(8, 8)(9, 8)(9, 9)(10, 9)(11, 9)(12,
    9)(12, 10)(13, 10)(14, 10)(15, 10)(15, 9)(14, 9)(13, 9)(13, 8)(12,
    8)(11, 8)(10, 8)(10, 7)(9, 7)(9, 6)\pspolygon(1, 2)(0, 2)(0, 3)(1,
    3)\pspolygon(4, 4)(3, 4)(3, 5)(3, 6)(4, 6)(5, 6)(5, 7)(6, 7)(7,
    7)(8, 7)(8, 6)(7, 6)(6, 6)(6, 5)(5, 5)(4, 5)(4, 4)\pspolygon(3,
    4)(2, 4)(2, 5)(3, 5)\pspolygon(8, 7)(7, 7)(7, 8)(7, 9)(8, 9)(8,
    10)(9, 10)(10, 10)(11, 10)(12, 10)(12, 9)(11, 9)(10, 9)(9, 9)(9,
    8)(8, 8)(8, 7)\pspolygon(7, 7)(6, 7)(6, 8)(7, 8)\pspolygon(1,
    3)(0, 3)(0, 4)(1, 4)\pspolygon(2, 4)(1, 4)(1, 5)(2,
    5)\pspolygon(3, 5)(2, 5)(2, 6)(2, 7)(3, 7)(4, 7)(5, 7)(5, 6)(4,
    6)(3, 6)(3, 5)\pspolygon(6, 7)(5, 7)(5, 8)(6, 8)\pspolygon(7,
    8)(6, 8)(6, 9)(7, 9)\pspolygon(8, 9)(7, 9)(7, 10)(8,
    10)\pspolygon(5, 7)(4, 7)(4, 8)(5, 8)\pspolygon(6, 8)(5, 8)(5,
    9)(6, 9)\pspolygon(7, 9)(6, 9)(6, 10)(7, 10)\pspolygon(1, 4)(0,
    4)(0, 5)(1, 5)\end{pspicture}
  \caption{An example of $2$-Dyck tiling.}
  \label{fig:2-DT}
\end{figure}

For $D\in \kD(\lm)$, there are two natural statistics: the number
$\tiles(D)$ of tiles in $D$ and the area $\area(D)$ of the region
occupied by $D$.  We also define
\[
\art_k(D) = \frac{k\cdot \area(D) + \tiles(D)}{k+1}.
\]

\begin{defn}
  For an up step $u$ of a $k$-Dyck path, we define the \emph{height}
  $\HT(u)$ of $u$ to be the number of squares between $u$ and the line
  $y=x/k$ plus 1.  A \emph{$k$-Hermite history} is a $k$-Dyck path in
  which every up step $u$ is labeled with an integer in
  $\{0,1,2,\dots,\HT(u)-1\}$. See for example Figure~\ref{fig:ht}.
  Given a $k$-Dyck path $\mu$, we denote by $\kH(\mu)$ the set of
  $k$-Hermite histories on $\mu$. The \emph{weight} $\wt(H)$ of a
  $k$-Hermite history is the sum of the labels in $H$.
\end{defn}

\begin{figure}
  \centering
  \begin{pspicture}(8,4) \kdyckgrid{8}{4}
\psline(0,0)(0, 1)(0, 2)(0, 3)(1, 3)(2, 3)(2, 4)(3, 4)(4, 4)(5, 4)(6,
4)(7, 4)(8, 4)
\psset{linewidth=.1,fillcolor=lightgray,fillstyle=solid}
\psframe(0,2)(1,3) \psframe(1,2)(2,3) \psframe(2,2)(3,3) \psframe(3,2)(4,3)
\rput[r](-.3,2.5){$u$}
  \end{pspicture}\qquad \qquad
  \begin{pspicture}(8,4) \kdyckgrid{8}{4}
\psline(0,0)(0, 1)(0, 2)(0, 3)(1, 3)(2, 3)(2, 4)(3, 4)(4, 4)(5, 4)(6,
4)(7, 4)(8, 4)
\rput[r](-.3,2.5){$0$}
\rput[r](-.3,1.5){$1$}
\rput[r](-.3,0.5){$0$}
\rput[r](1.7,3.5){$4$}
  \end{pspicture}
  \caption{The up step $u$ on the left has height $\HT(u)=5$ because
    there are $4$ squares between $u$ and the line $y=x/2$. The
    diagram on the right is an example of $2$-Hermite history.}
  \label{fig:ht}
\end{figure}
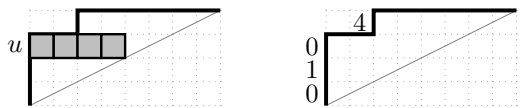

The following theorem is a generalization of \eqref{eq:KW_mu}.  Our
proof generalizes the bijective proof in \cite{JSK_DT}.

\begin{thm} \label{thmmuqint}
For $\mu\in\kDyck(n)$, we have
\begin{equation}
  \label{eq:0}
\sum_{D\in\kD(*/\mu)} q^{\tiles(D)} 
= \prod_{u\in \UP(\mu)} \qint{\HT(u)}.
\end{equation}
\end{thm}
\begin{proof}

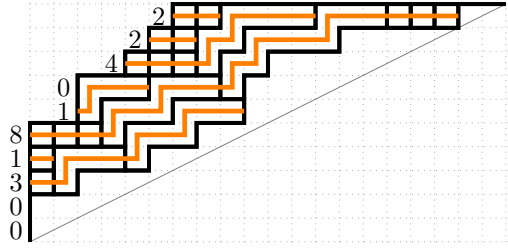
\begin{figure}
  \centering
\begin{pspicture}(0,0)(20,10) \kdyckgrid{20}{10}
 \psline(0,0)(0, 1)(0, 2)(0, 3)(0, 4)(0,
    5)(1, 5)(2, 5)(2, 6)(2, 7)(3, 7)(4, 7)(4, 8)(5, 8)(5, 9)(6, 9)(6,
    10)(7, 10)(8, 10)(9, 10)(10, 10)(11, 10)(12, 10)(13, 10)(14,
    10)(15, 10)(16, 10)(17, 10)(18, 10)(19, 10)(20, 10)\psline(0,0)(0,
    1)(0, 2)(1, 2)(2, 2)(2, 3)(3, 3)(4, 3)(5, 3)(5, 4)(6, 4)(7, 4)(7,
    5)(8, 5)(9, 5)(9, 6)(9, 7)(10, 7)(10, 8)(11, 8)(12, 8)(13, 8)(13,
    9)(14, 9)(15, 9)(16, 9)(17, 9)(18, 9)(18, 10)(19, 10)(20,
    10)\pspolygon(18, 9)(17, 9)(17, 10)(18, 10)\pspolygon(17, 9)(16,
    9)(16, 10)(17, 10)\pspolygon(5, 3)(4, 3)(4, 4)(4, 5)(5, 5)(6,
    5)(6, 6)(7, 6)(8, 6)(9, 6)(9, 5)(8, 5)(7, 5)(7, 4)(6, 4)(5, 4)(5,
    3)\pspolygon(16, 9)(15, 9)(15, 10)(16, 10)\pspolygon(2, 2)(1,
    2)(1, 3)(1, 4)(2, 4)(3, 4)(4, 4)(4, 3)(3, 3)(2, 3)(2,
    2)\pspolygon(9, 6)(8, 6)(8, 7)(8, 8)(9, 8)(9, 9)(10, 9)(11, 9)(12,
    9)(12, 10)(13, 10)(14, 10)(15, 10)(15, 9)(14, 9)(13, 9)(13, 8)(12,
    8)(11, 8)(10, 8)(10, 7)(9, 7)(9, 6)\pspolygon(1, 2)(0, 2)(0, 3)(1,
    3)\pspolygon(4, 4)(3, 4)(3, 5)(3, 6)(4, 6)(5, 6)(5, 7)(6, 7)(7,
    7)(8, 7)(8, 6)(7, 6)(6, 6)(6, 5)(5, 5)(4, 5)(4, 4)\pspolygon(3,
    4)(2, 4)(2, 5)(3, 5)\pspolygon(8, 7)(7, 7)(7, 8)(7, 9)(8, 9)(8,
    10)(9, 10)(10, 10)(11, 10)(12, 10)(12, 9)(11, 9)(10, 9)(9, 9)(9,
    8)(8, 8)(8, 7)\pspolygon(7, 7)(6, 7)(6, 8)(7, 8)\pspolygon(1,
    3)(0, 3)(0, 4)(1, 4)\pspolygon(2, 4)(1, 4)(1, 5)(2,
    5)\pspolygon(3, 5)(2, 5)(2, 6)(2, 7)(3, 7)(4, 7)(5, 7)(5, 6)(4,
    6)(3, 6)(3, 5)\pspolygon(6, 7)(5, 7)(5, 8)(6, 8)\pspolygon(7,
    8)(6, 8)(6, 9)(7, 9)\pspolygon(8, 9)(7, 9)(7, 10)(8,
    10)\pspolygon(5, 7)(4, 7)(4, 8)(5, 8)\pspolygon(6, 8)(5, 8)(5,
    9)(6, 9)\pspolygon(7, 9)(6, 9)(6, 10)(7, 10)\pspolygon(1, 4)(0,
    4)(0, 5)(1, 5) 
\psline[linecolor=orange, linewidth=.2](0,3.5)(1,3.5)
\psline[linecolor=orange, linewidth=.2](0,2.5)(1.5,2.5)(1.5,3.5)(4.5,3.5)(4.5,4.5)(6.5,4.5)(6.5,5.5)(9,5.5)
\psline[linecolor=orange,linewidth=.2](0,4.5)(3.5,4.5)(3.5,5.5)(5.5,5.5)(5.5,6.5)(8.5,6.5)
(8.5,7.5)(9.5,7.5)(9.5,8.5)(12.5,8.5)(12.5,9.5)(18,9.5)
\psline[linecolor=orange,linewidth=.2](2,5.5)(2.5,5.5)(2.5,6.5)(5,6.5)
\psline[linecolor=orange,linewidth=.2](4,7.5)(7.5,7.5)(7.5,8.5)(8.5,8.5)(8.5,9.5)(12,9.5)
\psline[linecolor=orange,linewidth=.2](5,8.5)(7,8.5)
\psline[linecolor=orange,linewidth=.2](6,9.5)(8,9.5)
\rput[r](-.3,.5){0}
\rput[r](-.3,1.5){0}
\rput[r](-.3,2.5){3}
\rput[r](-.3,3.5){1}
\rput[r](-.3,4.5){8}
\rput[r](1.7,5.5){1}
\rput[r](1.7,6.5){0}
\rput[r](3.7,7.5){4}
\rput[r](4.7,8.5){2}
\rput[r](5.7,9.5){2}
\end{pspicture}
\caption{An illustration of the bijection between $k$-Dyck tilings and
  $k$-Hermite histories.}
  \label{fig:2-DT-bij}
\end{figure}

  It suffices to find a bijection $f:\kD(*/\mu)\to\kH(\mu)$ such
  that if $f(D)=H$ then $\tiles(D)=\wt(H)$. We construct such a
  bijection as follows. Let $D\in \kD(*/\mu)$. In order to define the
  corresponding $f(D)=H\in\kH(\mu)$ we only need to define the
  labels of each up step in $\mu$. For a $k$-Dyck tile $\eta$, the
  \emph{entry} of $\eta$ is the left side of the lowest cell in $\eta$
  and the \emph{exit} of $\eta$ is the right side of the rightmost
  cell in $\eta$. For each up step $u$ in $\mu$, we travel the tiles
  in $D$ in the following way. If $u$ is the entry of a $k$-Dyck tile
  $\eta$ in $D$, then enter $\eta$ at the entry and leave $\eta$ from
  the exit. If the exit of $\eta$ is the entry of another $k$-Dyck
  tile of $D$ then travel that tile as well. Continue traveling in
  this way until we do not reach the entry of any $k$-Dyck tile in
  $D$. Then the label of $u$ is defined to be the number of tiles that we have
  traveled. See Figure~\ref{fig:2-DT-bij} for an example. It is not
  difficult to see $f(D)\in \kH(\mu)$. Clearly, we have
  $\tiles(D)=\wt(H)$.

  It remains to show that $f$ is a bijection. Let $H\in \kH(\mu)$.  We
  will find the inverse image $D=f^{-1}(H)$ recursively.  Suppose that
  $n$ is the length of $\mu$ and $m$ is the number of cells between
  $\mu$ and the line $y=x/k$.  If $n=0$ or $m=0$, then both
  $\kD(*/\mu)$ and $\kH(\mu)$ have a unique element, and $f^{-1}(H)$
  is the unique element in $\kD(*/\mu)$ without tiles. Now let
  $n,m\ge1$ and suppose that we can find $f^{-1}(H')$ for every
  $H'\in\kH(\mu')$ if the length of $\mu'$ is smaller than $n$ or the
  number of cells between $\mu'$ and the line $y=x/k$ is smaller than
  $m$. There are two cases.

  Case 1: $H$ has an up step with label $\ell\ge1$ followed by a down
  step. In this case let $H'$ be the $k$-Hermite history obtained from
  $H$ by exchanging the up step and the down step following it and
  decrease the label $\ell$ by $1$. Then the shape $\mu'$ of $H'$ has
  one less cells between $\mu'$ and the line $y=x/k$. By assumption we
  can find $f^{-1}(\mu')$. Then $f^{-1}(\mu)$ is the $k$-Dyck tiling
  obtained from $f^{-1}(\mu')$ by adding the single square
  $\mu\setminus\mu'$.

  Case 2: $H$ has no up step with label $\ell\ge1$ followed by a down
  step. Since $\mu$ is a $k$-Dyck path of positive length, we can find
  an up step $u$ followed by $k$ down steps. Since $u$ is followed by
  a down step, its label is $0$. Let $P$ be the point where $u$
  starts. Let $\mu'$ be the $k$-Dyck path obtained from $\mu$ by
  deleting $u$ and the $k$ down steps following $u$. By assumption, we
  can can find $f^{-1}(\mu')$. Then $f^{-1}(\mu)$ is the $k$-Dyck
  tiling obtained from $f^{-1}(\mu')$ by cutting it with the line of
  slope $-1$ passing through $P$ and inserting an up step followed by
  $k$ down steps. For each $k$-Dyck tile divided by the line, we
  attach the two divided pieces by connecting the separated points on
  the border with an up step and $k$ down steps following it.
See Figure~\ref{fig:stretch}.

\begin{figure}
  \centering
\begin{pspicture}(0,0)(12,6)
\kdyckgrid{12}{6}
\psline(0,0)(0, 1)(0, 2)(1, 2)(1, 3)(1, 4)(2, 4)(2, 5)(2, 6)(3, 6)(4, 6)(5, 6)(6, 6)(7, 6)(8, 6)(9, 6)(10, 6)(11, 6)(12, 6)\psline(0,0)(0, 1)(0, 2)(1, 2)(2, 2)(2, 3)(3, 3)(4, 3)(5, 3)(5, 4)(6, 4)(7, 4)(7, 5)(7, 6)(8, 6)(9, 6)(10, 6)(11, 6)(12, 6)\pspolygon(5, 3)(4, 3)(4, 4)(4, 5)(5, 5)(6, 5)(7, 5)(7, 4)(6, 4)(5, 4)(5, 3)\pspolygon(2, 2)(1, 2)(1, 3)(1, 4)(2, 4)(3, 4)(4, 4)(4, 3)(3, 3)(2, 3)(2, 2)\pspolygon(7, 5)(6, 5)(6, 6)(7, 6)\pspolygon(4, 4)(3, 4)(3, 5)(3, 6)(4, 6)(5, 6)(6, 6)(6, 5)(5, 5)(4, 5)(4, 4)\pspolygon(3, 4)(2, 4)(2, 5)(3, 5)\pspolygon(3, 5)(2, 5)(2, 6)(3, 6)
\rput[r](-.3,.5){0}
\rput[r](-.3,1.5){0}
\rput[r](.7,2.5){2}
\rput[r](.7,3.5){0}
\rput[r](1.7,4.5){3}
\rput[r](1.7,5.5){1}
\psline[linecolor=red](2,6)(5,3)
\end{pspicture} \qquad
$\rightarrow$ \qquad
\begin{pspicture}(0,0)(14,7)
\kdyckgrid{14}{7}
\psline(0,0)(0, 1)(0, 2)(1, 2)(1, 3)(1, 4)(2, 4)(2, 5)(2, 6)(2, 7)(3, 7)(4, 7)(5, 7)(6, 7)(7, 7)(8, 7)(9, 7)(10, 7)(11, 7)(12, 7)(13, 7)(14, 7)\psline(0,0)(0, 1)(0, 2)(1, 2)(2, 2)(2, 3)(3, 3)(4, 3)(5, 3)(5, 4)(6, 4)(7, 4)(7, 5)(8, 5)(9, 5)(9, 6)(9, 7)(10, 7)(11, 7)(12, 7)(13, 7)(14, 7)\pspolygon(5, 3)(4, 3)(4, 4)(4, 5)(5, 5)(6, 5)(6, 6)(7, 6)(8, 6)(9, 6)(9, 5)(8, 5)(7, 5)(7, 4)(6, 4)(5, 4)(5, 3)\pspolygon(2, 2)(1, 2)(1, 3)(1, 4)(2, 4)(3, 4)(4, 4)(4, 3)(3, 3)(2, 3)(2, 2)\pspolygon(9, 6)(8, 6)(8, 7)(9, 7)\pspolygon(4, 4)(3, 4)(3, 5)(3, 6)(4, 6)(5, 6)(5, 7)(6, 7)(7, 7)(8, 7)(8, 6)(7, 6)(6, 6)(6, 5)(5, 5)(4, 5)(4, 4)\pspolygon(3, 4)(2, 4)(2, 5)(3, 5)\pspolygon(3, 5)(2, 5)(2, 6)(2, 7)(3, 7)(4, 7)(5, 7)(5, 6)(4, 6)(3, 6)(3, 5)
\rput[r](-.3,.5){0}
\rput[r](-.3,1.5){0}
\rput[r](.7,2.5){2}
\rput[r](.7,3.5){0}
\rput[r](1.7,4.5){3}
\rput[r](1.7,5.5){1}
\rput[r](1.7,6.5){0}
\psline[linecolor=red](2,6)(5,3)
\psline[linecolor=red](4,7)(7,4)
\end{pspicture}  
  \caption{An example of cutting a $k$-Dyck tiling and inserting an up
  step and $k$ down steps. }
  \label{fig:stretch}
\end{figure}

  This gives the inverse map of $f$.
\end{proof}

It seems unlikely that there is a hook length formula for
$\kD(\lambda/*)$ when we fix the lower path $\lambda$ to be arbitrary.
If $n=6$, $k=2$, and $\lambda$ is the following path, then
we have $|\kD_n(\lambda/*)|=607$, a prime number.
\[
\begin{pspicture}(0,0)(12,6)
\kdyckgrid{12}{6}
\psline(0,0)(0, 1)(1, 1)(2, 1)(2, 2)(2, 3)(3, 3)(4, 3)(4, 4)(4, 5)(5, 5)(6, 5)(7, 5)(8, 5)(8, 6)(9, 6)(10, 6)(11, 6)(12, 6)
\end{pspicture}
\]
Also if $|\kD_n(\lambda/*)|=71$ for the following $\lambda$ with $n=5,k=2$.
\[
\begin{pspicture}(0,0)(10,5)
\kdyckgrid{10}{5}
\psline(0,0)(0, 1)(1, 1)(2, 1)(2, 2)(2, 3)(3, 3)(4, 3)(4, 4)(4, 5)(5, 5)(6, 5)(7, 5)(8, 5)(9, 5)(10, 5)
\end{pspicture}
\]

However, when
$\lambda$ is a zigzag path there is a nice generalization of
\eqref{eq:KW_lambda}.

\begin{thm}\label{thm:zigzag}
  Let $\lambda$ be the path $u^{n_1} d^{kn_1} u^{n_2} d^{kn_2}\cdots
u^{n_\ell} d^{kn_\ell}$, where $u$ means an up step and $d$ means a
  down step.  Then we have
\[
\sum_{D\in\kD(\lambda/*)} q^{\art_k(D)} 
=\qbinom{kn_1+n_2}{n_2}\qbinom{k(n_1+n_2)+n_3}{n_3}\cdots
\qbinom{k(n_1+\cdots+n_{\ell-1}) +n_\ell}{n_\ell},
\]
where $\qbinom{a}{b} = \frac{[a]_q!}{[b-a]_q![b]_q!}$. 
\end{thm}
\begin{proof}
  This can be proved using the same idea in the inductive proof in
  \cite{JSK_DT}. We will omit the details.
\end{proof}

\begin{problem}
Find a bijective proof of Theorem~\ref{thm:zigzag}.  
\end{problem}

\section{\texorpdfstring{$k$-Stirling permutations and the $k$-Bruhat order}{k-Stirling permutations and the k-Bruhat order}}

In this section we consider $k$-Stirling permutations which were
introduced by Gessel and Stanley \cite{Gessel1978} for $k=2$ and
studied further for general $k$ by Park \cite{Park1994}. 

A \emph{$k$-Stirling permutation of size $n$} is a permutation of the
multiset $\{1^k, 2^k,\dots, n^k\}$ such that if an integer $j$ appears
between two $i$'s then $i>j$. Let $\SP_n$ denote the set of
$k$-Stirling permutations of size $n$. We can represent a
$k$-Stirling permutation $\pi=\pi_1\pi_2\dots\pi_{kn}$ as the $n\times
kn$ matrix $M=(M_{i,j})$ defined by $M_{i,j} = 1$ if $\pi_j=i$ and
$M_{i,j}=0$ otherwise. Then the entries of $M$ are $0$'s and $1$'s
such that each column contains exactly one $1$, each row contains $k$
$1$'s, and it does not contain the following submatrix:
\[
\begin{pmatrix}
  1&0&1\\0&1&0
\end{pmatrix}.
\]
For example, see Figure~\ref{fig:k-inv}. 

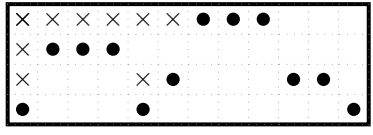
\begin{figure}
  \centering
   \psset{unit=4mm}
  \begin{pspicture}(12,-4)
    \psframe(0,0)(12,-4)
    \psgrid[gridcolor=gray,griddots=4,subgriddiv=0,gridlabels=0](0,0)(12,-4)
\rput(-.5,.5){\psdots(1,-4)(2,-2)(3,-2)(4,-2)(5,-4)(6,-3)(7,-1)(8,-1)(9,-1)
(10,-3)(11,-3)(12,-4) 
\cross(1,-1) \cross(1,-2) \cross(1,-3) 
\cross(1,-1) \cross(2,-1) \cross(3,-1) \cross(4,-1) \cross(5,-1)
\cross(6,-1) 
%\cross(2,-3) %\cross(3,-3) \cross(4,-3) 
\cross(5,-3)
}
  \end{pspicture}
  \caption{The matrix corresponding to 
    $422243111334 \in \mathfrak{S}_4^{(3)}$, where ones are represented as dots. The
    $k$-inversions are the cells with crosses.}
  \label{fig:k-inv}
\end{figure}

A \emph{$k$-inversion} of $\pi\in \SP_n$ is a pair $(i,j)\in
[n]\times[kn]$ such that $\pi_j > i$ and the first $i$ appears after
$\pi_j$. Equivalently, we will think of a $k$-inversion as an entry
(or a cell) in the matrix of $\pi$ which has $k$ $1$'s to the right in
the same row and one $1$ below in the same column, see
Figure~\ref{fig:k-inv}. We denote the set of $k$-inversions of $\pi$
by $\INV_k(\pi)$, and $\inv_k(\pi)=|\INV_k(\pi)|$.

\begin{prop} \label{invprod}
We have  
\[
\sum_{\pi\in \kS_n} q^{\inv_k(\pi)} =[k+1]_q[2k+1]_q\cdots[(n-1)k+1]_q.
\]
\end{prop}

\begin{proof}
This is an easy induction. Note that all the 1's in a Stirling permutation form a block of consecutive letters.
Starting from $\sigma\in\mathfrak{S}_{n-1}^{(k)}$, we build a Stirling permutation in $\mathfrak{S}_{n}^{(k)}$
by increasing all letters of $\sigma$ by 1, and inserting a block $1^k$ at some position. 
The positions where we can insert this block give the factor $[(n-1)k+1]_q$.
\end{proof}

As we can see in the following lemma, a $k$-Stirling permutation is
determined by its $k$-inversions.

\begin{lem}
  For $\sigma,\pi\in\kS_n$, if $\INV_k(\sigma)=\INV_k(\pi)$, we have
  $\sigma=\pi$.
\end{lem}
\begin{proof}
  Suppose that $\sigma\ne \pi$. Let $r$ be the smallest index
  satisfying $\sigma_r\ne \pi_r$. We can assume that
  $\sigma_r<\pi_r$. Let $m=\sigma_r$. Note that $\pi_j=m$ for some
  $j>r$. Then $\pi$ does not have the integer $m$ in the first $r$
  positions because otherwise $\pi_i=m$ for $i<r$ and we get $\pi_i
  <\pi_r >\pi_j$ which is forbidden. Then $(m,r)\in \INV_k(\pi)$ but
  $(m,r)\not\in\INV_k(\sigma)$, which is a contradiction. Thus we have
  $\sigma=\pi$.
\end{proof}

We are now ready to define the $k$-Bruhat order on $k$-Stirling
permutations.

\begin{defn}\label{def:Bruhat}
We define the \emph{$k$-Bruhat order} on $\SP_n$ given by the cover
relation $\sigma\lessdot\pi$ if $\pi$ is obtained from $\sigma$ by exchanging
the two numbers in positions $a_1$ and $a_{k+1}$ for some integers
$a_1<a_2<\dots<a_{k+1}$ satisfying the following conditions:
\begin{enumerate}
\item $\sigma_{a_1}=\sigma_{a_2}=\dots=\sigma_{a_k}<\sigma_{a_{k+1}}$, and
\item for all $a_k<i<a_{k+1}$ we have either $\sigma_i<\sigma_{a_1}$
  or $\sigma_i>\sigma_{a_{k+1}}$.
\end{enumerate}
In fact, if $k\ge2$, then for all $a_1<i<a_{k+1}$ with $i\ne a_j$ we
always have $\sigma_i<\sigma_{a_1}$, see Lemma~\ref{lem:bruhat}.
\end{defn}

Note that the $1$-Bruhat order is the usual Bruhat order. 
Figure~\ref{kbruhat} illustrates the $k$-Bruhat order.

\begin{figure} \centering
\psset{unit=4mm}
\begin{pspicture}(-5,0)(5,12)
                \rput(0,12){332211}
\psline(-0.2,11.5)(-2,10.5)\psline(0.2,11.5)(2,10.5)
       \rput(-2,10){332112}      \rput(2,10){322311}
\psline(-2.2,9.5)(-4,8.5)\psline(-1.8,9.5)(0,8.5)\psline(1.8,9.5)(0.2,8.5)\psline(2.2,9.5)(4,8.5)
\rput(-4,8){331122}    \rput(0,8){322113} \rput(4,8){223311}
\psline(-4,7.5)(-4,6.5)\psline(0,7.5)(0,6.5)\psline(4,7.5)(4,6.5)
%%%%%%%%%%%%%%%%%%%%%%%%%%%%%%%%%%%%%%%%%%%%%%%%%%%%%%%%%%
\psline(-3.6,7.5)(-0.4,6.5)\psline(0.4,7.5)(3.6,6.5)
%%%%%%%%%%%%%%%%%%%%%%%%%%%%%%%%%%%%%%%%%%%%%%%%%%%%%%%%%%
\rput(-4,6){311322}    \rput(0,6){321123} \rput(4,6){223113}
\psline(-4,5.5)(-4,4.5)   \psline(0,5.5)(0,4.5)  \psline(4,5.5)(4,4.5)
%%%%%%%%%%%%%%%%%%%%%%%%%%%%%%%%%%%%%%%%%%%%%%%%%%%%%%%%%%
\psline(-3.6,5.5)(-0.4,4.5)\psline(0.4,5.5)(3.6,4.5)
%%%%%%%%%%%%%%%%%%%%%%%%%%%%%%%%%%%%%%%%%%%%%%%%%%%%%%%%%%
\rput(-4,4){113322}    \rput(0,4){311223} \rput(4,4){221133}
\psline(-4,3.5)(-2.2,2.5)\psline(-0.2,3.5)(-1.8,2.5)\psline(0.2,3.5)(1.8,2.5)\psline(4,3.5)(2.2,2.5)
        \rput(-2,2){113223}      \rput(2,2){211233}
\psline(-2,1.5)(-0.2,0.5) \psline(2,1.5)(0.2,0.5)
                 \rput(0,0){112233}
\end{pspicture}
 \caption{The Hasse diagram of $\mathfrak{S}_3^{(2)}$. \label{kbruhat} }
\end{figure}
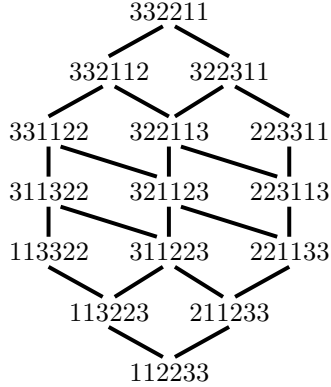

\begin{remark}
  The elements of $\mathfrak{S}_n^{(k)}$ where all occurrences of $i$
  are consecutive for any $i$, form a subset which is in natural
  bijection with $\Sym_n$.  In general ($k>1$ and $n>2$), the
  induced order on this subset is strictly contained in the Bruhat
  order, and strictly contains the left weak order.
\end{remark}

\begin{lem}\label{lem:bruhat}
  In this lemma we use the notation in Definition~\ref{def:Bruhat}.
  If $k\ge2$, then we have $\sigma_i<\sigma_{a_1}$ for all
  $a_1<i<a_{k+1}$ with $i\ne a_j$.
\end{lem}
\begin{proof}
  By the definition of $k$-Stirling permutation, for all $a_1< i<
  a_{k}$ with $i\ne a_j$ we have $\sigma_i < \sigma_k$, and for all
  $a_2<i<a_{k+1}$ with $i\ne a_j$ we have $\pi_i <
  \pi_k=\sigma_k$. Thus we have $\sigma_i=\pi_i<\sigma_k$ for all
  $a_1<i<a_{k+1}$ with $i\ne a_j$.
\end{proof}

\begin{lem}\label{lem:cover}
  Let $\sigma\lessdot\pi$ in $\kS_n$. Then $\INV_k(\pi)$ is obtained
  from $\INV_k(\sigma)$ by adding one cell and moving some cells
  (maybe none) to the west or north. Moreover, if a cell is moved to
  the west (respectively, north), then the new location of the cell is
  south (respectively, to east) of the newly added cell in the same
  column (respectively, row).
\end{lem}
\begin{proof}
  Suppose that $\pi$ is obtained from $\sigma$ as described in
  Definition~\ref{def:Bruhat}. 
  Then $\INV_k(\pi)$ is obtained from $\INV_k(\sigma)$ as follows.
  We add the inversion $(\sigma_{a_1},a_1)$, and change each $k$-inversion 
  of the form $(r,a_{k+1})$ for some $\sigma_{a_{k+1}} < r < \sigma_{a_1}$ to $(r,a_1)$.
  Furthermore if $k=1$, we change each $k$-inversion of the form $(\sigma_{a_{k+1}},j)$ for some $j<a_{k+1}$ to
  $(\sigma_{a_1},j)$.
\end{proof}

We note that in Lemma~\ref{lem:cover}, moving cells to north can happen
only when $k=1$.

By Lemma~\ref{lem:cover} we have $\inv_k(\pi)=\inv_k(\sigma)+1$ if
$\sigma\lessdot\pi$. This proves:

\begin{prop}
  Endowed with the $k$-Bruhat order, $\kS_n$ is a graded poset with rank function $\inv_k$.
\end{prop}

The following generalization of 132-avoiding permutations is straightforward and natural.
\begin{defn}
  A $k$-Stirling permutation $\sigma\in \kS_n$ is 132-avoiding if there is no $1\leq i<j<k \leq kn$
  such that $\sigma_i<\sigma_k<\sigma_j$. Let $\kS_n(132)$ denote the set of 132-avoiding $k$-Stirling permutations in $\kS_n$.
\end{defn}

In a $k$-Dyck tilings without tiles, the lower path and upper path are the same.
So these tilings are trivially in bijection with $k$-Dyck paths. 
As for the $k$-Stirling permutations, we have:

\begin{prop}\label{prop:132_k}
  The inversions of a 132-avoiding $k$-Stirling permutation $\sigma$
  are arranged as the cells of a top-left justified Ferrers
  diagram. Define a path $\alpha(\sigma)$ from the bottom left to the
  top right corner, and following the boundary of this Ferrers diagram
  with up and right steps.  Then $\alpha$ is a bijection between
  132-avoiding $k$-Stirling permutations and $k$-Dyck paths of length
  $n$, in particular both are counted by the Fuss-Catalan numbers
  $\frac{1}{kn+1}\binom{(k+1)n}{n}$.
\end{prop}

The proof is simple and similar to the case $k=1$. For example, see
Figure~\ref{fig:alpha_k}. 

\begin{prop} \label{pisigmaincl}
  Suppose that $\sigma\in\kS_n$ is $132$-avoiding.  Then, for
  $\pi\in\kS_n$, we have $\sigma\le\pi$ if and only if
  $\INV_k(\sigma)\subseteq \INV_k(\pi)$.
\end{prop}
\begin{proof}
  Since $\sigma$ is $132$-avoiding, $\INV(\sigma)$ is a Ferrers
  diagram, say $\lambda$.  By Lemma~\ref{lem:cover} if $\tau\lessdot
  \rho$ and $\lambda\subset \INV_k(\tau)$, then we also have
  $\lambda\subset\INV_k(\rho)$. This implies the ``only if'' part.

  We will prove the ``if'' part by induction on the number $m =
  \inv_k(\pi)-\inv_k(\sigma)$. If $m=0$, it is true. Suppose $m>0$.
  Let $(i,j)$ be an element in $\INV_k(\pi)\setminus\INV_k(\sigma)$
  with $j$ as large as possible. Then there are integers
  $a_1<a_2<\dots<a_k$ such that
  $\pi_{a_1}=\pi_{a_2}=\dots=\pi_{a_k}=i$ and $a_1>j$. Then we have
  $\pi_t<i$ for all $j<t<a_1$ by the maximality of $j$. Let $\pi'$ be
  obtained from $\pi$ by exchanging the two integers in positions $j$
  and $a_k$. Then $\pi'\lessdot\pi$. By Lemma~\ref{lem:cover}
  $\INV_k(\pi')$ is obtained from $\INV_k(\pi)$ by removing the cell
  $(i,j)$ and moving some cell located south or east of $(i,j)$. Thus
  $\INV_k(\pi')$ still contains $\INV_k(\sigma)$ which is a Ferrers
  diagram. By induction we have $\sigma\le \pi'$, which completes the
  proof.
\end{proof}

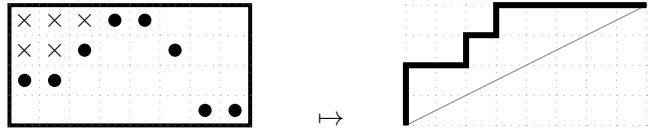
\begin{figure}
  \centering
   \psset{unit=4mm}
  \begin{pspicture}(8.2,4)
    \psframe(0,0)(8,4)
    \psgrid[gridcolor=gray,griddots=4,subgriddiv=0,gridlabels=0](0,0)(8,4)
    \psdots(0.5,1.5)(1.5,1.5)(2.5,2.5)(3.5,3.5)(4.5,3.5)(5.5,2.5)(6.5,0.5)(7.5,0.5)
    \rput(0.5,3.5){$\times$}
    \rput(0.5,2.5){$\times$}
    \rput(1.5,3.5){$\times$}
    \rput(1.5,2.5){$\times$}
    \rput(2.5,3.5){$\times$}
  \end{pspicture}  \qquad $\mapsto$ \qquad
 \begin{pspicture}(8,4)
%  \psgrid[gridcolor=gray,griddots=4,subgriddiv=0,gridlabels=0](0,0)(6,3)
%  \psframe(0,0)(6,3)  \psline[linecolor=lightgray](0,0)(6,3)
  \kdyckgrid{8}{4}
  \psline[linewidth=0.8mm](0,0)(0, 2)(2, 2)(2, 3)(3, 3)(3, 4)(8, 4)
 \end{pspicture}
  \caption{The bijection $\alpha$ in Proposition~\ref{prop:132_k}.}
  \label{fig:alpha_k}
\end{figure}

\begin{prop}  \label{bijinthist}
 Let $\sigma\in\kS_n(132)$, and $\mu=\alpha(\sigma)$ the corresponding $k$-Dyck path.
 There is a bijection between the interval $\{ \pi : \pi\geq\sigma \}$ in $\kS_n$ and $k$-Hermite histories 
 of shape $\mu$. It is such that $\inv_k(\pi) - \inv_k(\sigma)$ is sent to the sum of weights in the Hermite
 history. 
\end{prop}

\begin{proof}
  Let $\pi\in\kS_n(132)$ with $\pi\geq\sigma$.  Note that by
  Proposition~\ref{pisigmaincl}, the inversions of $\sigma$ are
  inversions of $\pi$.  Then we define a $k$-Hermite history as
  follows.  For $i\leq n$, the label of the up step in the $i$th row
  from the top is the number of $\times$'s in the matrix of $\sigma$
  that are in the $i$th row, and below $\mu$.  Since $\inv_k(\sigma)$
  is the number of $\times$'s above the path $\mu$ by definition of the
  bijection in Figure~\ref{fig:alpha_k}, the number $\inv_k(\pi) -
  \inv_k(\sigma)$ is the sum of weights in the Hermite history. So it
  remains only to show that these labels define a Hermite history
  (i.e. they fall in the right range) and that this is a bijection.

  Let us start with the first row. If the height of the up step in the
  first row is $h$, then there are $h+k$ cells to the right of the up
  step in the first row. Since there are $k$ consecutive dots in the
  first row in $\pi$, the $\times$'s are located in the cells after
  the up step and before the $k$ consecutive dots. Thus the number of
  $\times$'s in the first row is among $0,1,2,\dots,h$. Now consider
  the second row.  If the height of the up step in the first row is
  $h'$, then there are $h'+2k$ cells to the right of the up step in
  the second row. By the condition for $k$-Stirling permtutation, if
  we delete the columns which have a dot in the first row, then the
  dots in the second row are consecutive, and the $\times$'s are
  located in the cells after the up step and before the $k$
  consecutive dots. Thus the number of $\times$'s in the second row is
  among $0,1,2,\dots,h'$. In this manner, we can see that the number
  of $\times$'s in each row is at most the height of the up step in
  the same row. This argument also shows that once the number of
  $\times$'s in each row is determined, we can uniquely construct the
  corresponding permutation $\pi$ by putting dots starting from the
  first row. This implies that the map is a desired bijection.
\end{proof}

We now have a generalization of
Proposition~\ref{fixedupperpathbruhat}. 
It is a rewriting of Theorem~\ref{thmmuqint} using the bijection from Proposition \ref{bijinthist}.

\begin{thm}
If $\mu$ is a $k$-Dyck path corresponding to $\sigma\in
\kS_n(132)$, then
\[
\sum_{D\in \kD_n(*/\mu)} q^{\tiles(D)} = 
\sum_{\pi\ge\sigma} q^{\inv_k(\pi)-\inv_k(\sigma)}.
\]
\end{thm}

\section{\texorpdfstring{$k$-regular noncrossing partitions}{k-regular noncrossing partitions}}

In this section, we take another point of view on the $k$-Stirling
permutations studied in the previous section.

\begin{defn}
We denote by $\NC_n^{(k)}$ the set of $k$-{\it regular noncrossing partitions} of size $n$, i.e. set partitions 
of $[kn]$ such that each block contains $k$ elements ($k$-regular), and there are no
integers $a<b<c<d$ such that $a$, $c$ are in one block, and $b$, $d$ in another block (noncrossing).
To each $k$-regular noncrossing partition $\pi$ of $[kn]$, we define its {\it nesting poset} $\Nest(\pi)$
as follows: the elements of the poset are the blocks of $\pi$, and $x\leq y$ in the poset when 
the block $x$ lies between two elements of the block $y$.
\end{defn}

There is a natural way to consider a $k$-Stirling permutation as a
linear extension of the nesting poset of a $k$-regular noncrossing
partition.

\begin{prop} \label{bijNC}
There is a bijection between $\mathfrak{S}_n^{(k)}$ and pairs $(\pi,E)$ where 
$\pi \in \NC_n^{(k)}$ and $E$ is a linear extension of the poset $\Nest(\pi)$.
\end{prop}

\begin{proof}
Let $\sigma\in \mathfrak{S}_n^{(k)}$. We define $\pi$ by saying that 
$i$ and $j$ are in the same block if $\sigma_i=\sigma_j$, and $E$ is defined
by saying that the label of a block of $\pi$ is $\sigma_i$ where $i$ is any element 
of the block. We can see that $\pi$ is noncrossing from the definition of Stirling
permutations, and $E$ is a linear extension by definition of the nesting poset.
In the exemple $\sigma=422243111334$, we get the noncrossing partition
in Figure~\ref{exbijNC} where the labels define the linear extension of the nesting poset.
The inverse bijection is simple to describe: $\sigma_i$ is equal to $j$ if $i$ is in 
the block with label $j$.
\end{proof}

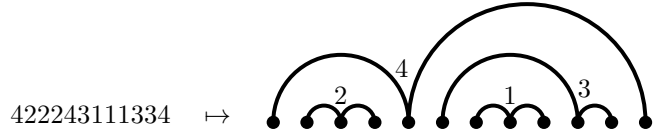
\begin{figure}\psset{unit=4.5mm}
422243111334 \quad
$\mapsto$
\begin{pspicture}(12,3.5)
   \psdots(1,0)(2,0)(3,0)(4,0)(5,0)(6,0)(7,0)(8,0)(9,0)(10,0)(11,0)(12,0)
   \psarc(2.5,0){0.5}{0}{180}\psarc(3.5,0){0.5}{0}{180}
   \psarc(7.5,0){0.5}{0}{180}\psarc(8.5,0){0.5}{0}{180}
   \psarc(8  ,0){2  }{0}{180}\psarc(10.5,0){0.5}{0}{180}
   \psarc(3  ,0){2  }{0}{180}\psarc(8.5,0){3.5}{0}{180}
   \rput(8,0.8){1}
   \rput(3,0.8){2}
   \rput(10.2,1){3}
   \rput(4.8,1.6){4}
   \end{pspicture}
 \caption{The bijection proving Proposition~\ref{bijNC}.\label{exbijNC}}
\end{figure}

The poset $\Nest(\pi)$ is always a forest, so it is possible to consider pairs $(\pi,E)$ as 
a decreasing forest. Each block of $\pi$ is a vertex of the forest, and the forest structure
is the order $\Nest(\pi)$. Also, the labelling $E$ naturally gives the decreasing labelling 
of the forest.

Let $b$ be a block of $\pi$, the elements of $b$ begin $i_1<\dots<i_k$. Then the descendants
of $b$ in the forest can be distinguished into $k-1$ categories, depending on the index $j$ such 
that a descendant of $b$ lies between $i_j$ and $i_{j+1}$.

So we arrive at the following definition.

\begin{defn}
Let $T_n^{(k-1)}$ denote the set of $(k-1)$-ary plane forests on $n$ vertices defined by the following conditions:
\begin{itemize}
 \item the descendants of a vertex have a structure of a $(k-1)$-uple of ordered lists,
 \item the vertices are labeled with integers from $1$ to $n$ and the labels are decreasing from
       the roots to the leaves.
\end{itemize}

\end{defn}

As a result of the above discussion, we obtain that there is a bijection between $\kS_n$ and $T_n^{(k-1)}$.
However we do not insist on this point of view, since the definition of the trees in $T_n^{(k-1)}$ is not 
particularly natural.

There is a hook length formula for the number of linear extensions of a forest \cite{Bjorner1991}. 
If $x\in\Nest(\pi)$, let $h_x$ denote the number of elements below $x$ in the 
nesting poset, then the number of linear extensions of $\pi\in \NC_n^{(k)}$ is
\[
  \frac{n!}{\prod_{x\in \Nest(\pi)} h_x }.
\]
This gives the following formula for the number of $k$-Stirling permutations.

\begin{prop}[Multifactorial hook length formula] %[$q$-multifactorial hook length formula]
\[
1(k+1)(2k+1)\dots((n-1)k+1) = \sum_{ \pi \in \NC_n^{(k)} }  \frac{ n!  }{  \prod\limits_{ x \in \Nest(\pi) } h_x  }.
\] 
\end{prop}

In the case $k>2$, it is not clear how to follow the $q$-statistic, but for $k=2$ we can use the bijection
between noncrossing matchings and plane forests (which is just $\pi\mapsto \Nest(\pi)$), and use the 
$q$-hook length formula from \cite{Bjorner1989}.
We get a hook length formula for $[1]_q[3]_q \dots [2n-1]_q$. 

\begin{prop}[$q$-double factorial hook length formula] %[$q$-multifactorial hook length formula]
We have
\[
  %\prod_{i=0}^{n-1} [ki+1]_q = \sum_{ \pi \in \NC_n^{(k)} }  \frac{ [n]_{q^k}!  }{  \prod\limits_{ x \in \Nest(\pi) }  [h_x]_{q^k}  } \times ???
  % Is there really a nice hook length formula for $k>2$?
  % 
  [1]_q[3]_q \dots [2n-1]_q = \sum_{F} [n]_{q^2}! \prod_{v \in F} \frac{ q^{h_v-1}  }{ [h_v]_{q^2} },
\] 
where the sum is over all  plane forests $F$ with $n$ vertices.
\end{prop}

\begin{proof}
We know that the left hand side is the inversion generating function of $\mathfrak{S}_n^{(2)}$.
It remains to understand what becomes the number of inversion through the bijection which 
send a 2-Stirling permutation to a noncrossing matching with labeled blocks, or equivalently
increasing plane forests.

To this end, we distinguish two kinds of inversions. Let $i<j$ and $\sigma\in\mathfrak{S}_n^{(2)}$.
If the four letters $i,i,j,j$ appear in $\sigma$ in this order, there is no inversion.
If they appear in the order $j,i,i,j$, there is one inversion, and this corresponds to 
the case where the vertex with label $i$ is below the vertex with label $j$.
It means we have to count one inversion for each pair of comparable vertices, and
the number of comparable vertices is clearly
$\sum_{ x \in \pi } (h_x-1)$. In particular, it does not depend on the labeling.
If they appear in the order $j,j,i,i$, there are two inversions, and this situation corresponds
to the case where the vertices with labels $i,j$ form an inversion in the forest.

So for a particular forest $F$, we get the term
\[
  [n]_{q^2}! \prod_{v \in F} \frac{ q^{h_v-1}  }{ [h_v]_{q^2} }.
\]
This completes the proof.
\end{proof}

\section{Symmetric Dyck tilings and marked increasing forests}

A \emph{symmetric plane forest} is a plane forest which is invariant
under the reflection about a line, called the \emph{center line}. A
\emph{center vertex} is a vertex on the center line.  The
\emph{left part} of a symmetric plane forest is the subforest
consisting of vertices on or to the left of the center line.  The
\emph{size} of a symmetric plane forest is the number of vertices in
the left part of it.
 
Let $F$ be a symmetric plane forest of size $n$.  A \emph{marked
  increasing labeling} of $F$ is a labeling of the left part of $F$
with $[n]$ such that the labels are increasing from roots to leaves,
each integer appears exactly once, and each non-center vertex may be
marked with $*$. See Figure~\ref{fig:sym_tree} for an example of a
marked increasing labeling. We denote the set of marked increasing
labelings of $F$ by $\INC^*(F)$.

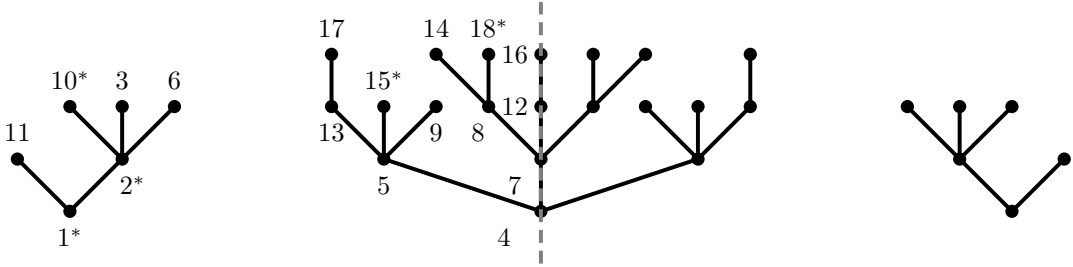
\begin{figure} \centering
\psset{unit=2.2}
\begin{pspicture}(-10,-2)(10,3)

\psdots(1,2)(2,1)(3,0)(3,1)(4,1)(4,2)(1,1)(2,2)
\psdots(-1,2)(-2,1)(-3,0)(-3,1)(-4,1)(-4,2)(-1,1)(-2,2)

\psdots(0,-1)(0,0)(0,1)(0,2)
\psline(0,-1)(0,0)
\psline(0,1)(0,0)
\psline(0,1)(0,2)

\psline(0,0)(1,1)
\psline(0,-1)(3,0)

\psline(1,2)(1,1)
\psline(1,1)(2,2)
\psline (3,0)(2,1)
\psline (3,0)(3,1)
\psline (3,0)(4,1)
\psline (4,1)(4,2)

\psline(0,0)(-1,1)
\psline(-1,2)(-1,1)
\psline(-1,1)(-2,2)
\psline (-3,0)(-2,1)
\psline (-3,0)(-3,1)
\psline (-3,0)(-4,1)
\psline (-4,1)(-4,2)
\psline(0,-1)(-3,0)

\rput(-1,2.5){$18^*$}
\rput(-2,2.5){$14$}
\rput(-2,0.5){$9$}
\rput(-1.2,0.5){$8$}
\rput(-3,-0.5){$5$}
\rput(-3,1.5){$15^*$}
\rput(-4,.5){$13$}
\rput(-4,2.5){$17$}

\rput(0,-1){
\psdots(7,2)(8,1)(8,2)(9,0)(9,2)(10,1)
\psline(7,2)(8,1)
\psline(8,1)(8,2)
\psline(8,1)(9,2)
\psline(8,1)(9,0)
\psline(9,0)(10,1)
}

\rput(-.7,-1.5){$4$}
\rput(-.5,-.5){$7$}
\rput(-.5,1){$12$}
\rput(-.5,2){$16$}

\rput(0,-1){
\psdots(-7,2)(-8,1)(-8,2)(-9,0)(-9,2)(-10,1)
\psline(-7,2)(-8,1)
\psline(-8,1)(-8,2)
\psline(-8,1)(-9,2)
\psline(-8,1)(-9,0)
\psline(-9,0)(-10,1)
\rput(-7,2.5){6}
\rput(-7.8,.5){$2^*$}
\rput(-8,2.5){$3$}
\rput(-9,-0.5){$1^*$}
\rput(-9,2.5){$10^*$}
\rput(-10,1.5){11}}

\psline[linecolor=gray, linestyle=dashed](0,-2)(0,3)
\end{pspicture}
\caption{A symmetric plane tree of size $18$ and a marked increasing
  labeling of it. The center line is the dashed line.}
\label{fig:sym_tree}
\end{figure}

Let $L\in \INC^*(F)$.  We denote by $\MARK(L)$ the set of labels of
the marked vertices in $L$. An \emph{inversion} of $L$ is a pair of
vertices $(u,v)$ such that $L(u)>L(v)$, $u$ and $v$ are incomparable,
and $u$ is to the left of $v$. Let $\INV(L)$ denote the set of
inversions of $L$.

A Dyck path $\lambda$ of length $2n$ is called \emph{symmetric} if it
is invariant under the reflection along the line $x+y= n$. For two
symmetric Dyck paths $\lambda$ and $\mu$ of length $2n$, a Dyck tiling
of $\lm$ is called \emph{symmetric} if it is invariant under the
reflection along the line $x+y= n$. See Figure~\ref{fig:sym_DT} for an
example of symmetric Dyck tiling. We denote by $\D_\sym(\lm)$ the set
of symmetric Dyck tilings of shape $\lm$.

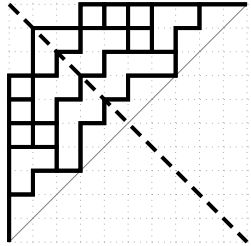
\begin{figure}
  \centering
\begin{pspicture}(0,0)(10,10)
\kdyckgrid{10}{10}
\psline(0,0)(0, 1)(0, 2)(0, 3)(0, 4)(0,
  5)(0, 6)(0, 7)(1, 7)(1, 8)(1, 9)(2, 9)(3, 9)(3, 10)(4, 10)(5, 10)(6,
  10)(7, 10)(8, 10)(9, 10)(10, 10)\psline(0,0)(0, 1)(0, 2)(1, 2)(1,
  3)(2, 3)(3, 3)(3, 4)(3, 5)(4, 5)(4, 6)(5, 6)(5, 7)(6, 7)(7, 7)(7,
  8)(7, 9)(8, 9)(8, 10)(9, 10)(10, 10)\pspolygon(3, 3)(2, 3)(2, 4)(2,
  5)(2, 6)(3, 6)(3, 7)(4, 7)(4, 8)(5, 8)(6, 8)(7, 8)(7, 7)(6, 7)(5,
  7)(5, 6)(4, 6)(4, 5)(3, 5)(3, 4)(3, 3)\pspolygon(7, 8)(6, 8)(6,
  9)(6, 10)(7, 10)(8, 10)(8, 9)(7, 9)(7, 8)\pspolygon(6, 8)(5, 8)(5,
  9)(6, 9)\pspolygon(6, 9)(5, 9)(5, 10)(6, 10)\pspolygon(1, 2)(0,
  2)(0, 3)(0, 4)(1, 4)(2, 4)(2, 3)(1, 3)(1, 2)\pspolygon(2, 4)(1,
  4)(1, 5)(2, 5)\pspolygon(2, 5)(1, 5)(1, 6)(1, 7)(2, 7)(2, 8)(3,
  8)(3, 9)(4, 9)(5, 9)(5, 8)(4, 8)(4, 7)(3, 7)(3, 6)(2, 6)(2,
  5)\pspolygon(5, 9)(4, 9)(4, 10)(5, 10)\pspolygon(4, 9)(3, 9)(3,
  10)(4, 10)\pspolygon(2, 7)(1, 7)(1, 8)(1, 9)(2, 9)(3, 9)(3, 8)(2,
  8)(2, 7)\pspolygon(1, 4)(0, 4)(0, 5)(1, 5)\pspolygon(1, 5)(0, 5)(0,
  6)(1, 6)\pspolygon(1, 6)(0, 6)(0, 7)(1, 7)
\psline[linestyle=dashed](0,10)(10,0)
\end{pspicture}
\caption{An example of symmetric Dyck tiling. It is symmetric with respect to
  the dashed line.}
\label{fig:sym_DT}
\end{figure}

For a symmetric Dyck tiling $D$, a \emph{positive tile} is a tile
which lies strictly to the left of the center line, and a \emph{zero
  tile} is a tile which intersects with the center line.  We denote by
$\tiles_+(D)$ and $\tiles_0(D)$ the number of positive tiles and zero
tiles, respectively. Note that the total number of tiles in $D$ is
$\tiles(D) = 2\cdot \tiles_+(D) + \tiles_0(D)$. We also define
$\area_+(D)$ and $\area_0(D)$ to be the total area of positive tiles
and zero tiles in $D$, respectively, and
\[
\art_+(D) = \frac{\area_+(D)+\tiles_+(D)}2,
\qquad \art_0(D) = \frac{\area_0(D)+\tiles_0(D)}2.
\]
For example, if $D$ is the symmetric Dyck tiling in
Figure~\ref{fig:sym_DT}, we have
$\tiles_+(D) = 5$, $\tiles_0(D) = 3$, $\area_+(D)=7$, $\area_0(D)=19$,
$\art_+(D)=(7+5)/2=6$, and $\art_0(D)=(19+3)/2=11$.

\begin{figure}
  \centering
\begin{pspicture}(-5,0)(5,2)
\psdots(-3,1)(3,1)
\rput(-4,1){$1^*$}
\psline[linecolor=gray, linestyle=dashed](0,0)(0,2)
\end{pspicture} \qquad $\leftrightarrow$ \qquad
\begin{pspicture}(0,0)(10,3)
\kdyckgrid{2}{2}
\psline(0,0)(0, 1)(0, 2)(1, 2)(2, 2) \psline(0,0)(0, 1)(1, 1)(1, 2)(2, 2)
\pspolygon[fillstyle=solid,fillcolor=lightgray](1, 1)(0, 1)(0, 2)(1, 2)
\psline[linestyle=dashed](0,2)(2,0)
\end{pspicture}

\begin{pspicture}(-5,0)(5,2)
\psdots(-3,1)(3,1)(0,1)
\rput(-4,1){$1^*$}
\rput(-1,1){$2$}
\psline[linecolor=gray, linestyle=dashed](0,0)(0,2)
\end{pspicture} \qquad $\leftrightarrow$ \qquad
\begin{pspicture}(0,0)(10,4)
\kdyckgrid{3}{3}\psline[linestyle=dashed](0,3)(3,0)
\psline(0,0)(0, 1)(0, 2)(0, 3)(1, 3)(2,
  3)(3, 3)\psline(0,0)(0, 1)(1, 1)(1, 2)(2, 2)(2, 3)(3,
  3)\pspolygon(1, 1)(0, 1)(0, 2)(0, 3)(1, 3)(2, 3)(2, 2)(1, 2)(1,
  1)\end{pspicture}

\begin{pspicture}(-5,0)(5,4)
\psdots(-3,1)(3,1)(0,1)(-3,3)(3,3)
\psline(-3,1)(-3,3)
\psline(3,1)(3,3)
\rput(-4,1){$1^*$}
\rput(-4,3){$3^*$}
\rput(-1,1){$2$}
\psline[linecolor=gray, linestyle=dashed](0,0)(0,4)
\end{pspicture} \qquad $\leftrightarrow$ \qquad
\begin{pspicture}(0,0)(10,6)
\kdyckgrid{5}{5}
\psline(0,0)(0, 1)(0, 2)(0, 3)(0, 4)(0,
  5)(1, 5)(2, 5)(3, 5)(4, 5)(5, 5)\psline(0,0)(0, 1)(0, 2)(1, 2)(2,
  2)(2, 3)(3, 3)(3, 4)(3, 5)(4, 5)(5, 5)\pspolygon(2, 2)(1, 2)(1,
  3)(1, 4)(2, 4)(3, 4)(3, 3)(2, 3)(2, 2)\pspolygon(3, 4)(2, 4)(2,
  5)(3, 5)\pspolygon(2, 4)(1, 4)(1, 5)(2, 5)\pspolygon(1, 2)(0, 2)(0,
  3)(1, 3)\pspolygon(1, 3)(0, 3)(0, 4)(1, 4)\pspolygon(1, 4)(0, 4)(0,
  5)(1, 5)
\psset{fillstyle=solid,fillcolor=lightgray}
\psframe(0,2)(1,3)
\psframe(0,3)(1,4)
\psframe(0,4)(1,5)
\psframe(1,4)(2,5)
\psframe(2,4)(3,5)
\psline[linestyle=dashed](0,5)(5,0)
\end{pspicture} 

\begin{pspicture}(-5,0)(5,4)
\psdots(-3,1)(3,1)(0,1)(-3,3)(3,3)(-1,3)(1,3)
\psline(-3,1)(-3,3)
\psline(3,1)(3,3)
\psline(0,1)(-1,3)
\psline(0,1)(1,3)
\rput(-4,1){$1^*$}
\rput(-4,3){$3^*$}
\rput(-1,4){$4$}
\rput(-1,1){$2$}
\psline[linecolor=gray, linestyle=dashed](0,0)(0,4)
\end{pspicture} \qquad $\leftrightarrow$ \qquad
\begin{pspicture}(0,0)(10,8)
\kdyckgrid{7}{7}
\psline(0,0)(0, 1)(0, 2)(0, 3)(0, 4)(0,
  5)(0, 6)(1, 6)(1, 7)(2, 7)(3, 7)(4, 7)(5, 7)(6, 7)(7, 7)
\psline(0,0)(0, 1)(0, 2)(1, 2)(2, 2)(2, 3)(2, 4)(3, 4)(3, 5)(4,
  5)(5, 5)(5, 6)(5, 7)(6, 7)(7, 7)\pspolygon(2, 2)(1, 2)(1,
  3)(1, 4)(1, 5)(2, 5)(2, 6)(3, 6)(4, 6)(5, 6)(5, 5)(4, 5)(3, 5)(3,
  4)(2, 4)(2, 3)(2, 2)\pspolygon(5, 6)(4, 6)(4, 7)(5, 7)\pspolygon(4,
  6)(3, 6)(3, 7)(4, 7)\pspolygon(1, 2)(0, 2)(0, 3)(1, 3)\pspolygon(1,
  3)(0, 3)(0, 4)(1, 4)\pspolygon(1, 4)(0, 4)(0, 5)(0, 6)(1, 6)(1,
  7)(2, 7)(3, 7)(3, 6)(2, 6)(2, 5)(1, 5)(1, 4)
\psline[linestyle=dashed](0,7)(7,0)
\end{pspicture} 
 
\begin{pspicture}(-5,0)(5,4)
\psdots(-3,1)(3,1)(0,1)(-4,3)(4,3)(-1,3)(1,3)(2,3)(-2,3)
\psline(-3,1)(-4,3)
\psline(3,1)(4,3)
\psline(0,1)(-1,3)
\psline(0,1)(1,3)
\psline(-3,1)(-2,3)
\psline(3,1)(2,3)
\rput(-4,1){$1^*$}
\rput(-4,4){$3^*$}
\rput(-2,4){$5$}
\rput(-1,4){$4$}
\rput(-1,1){$2$}
\psline[linecolor=gray, linestyle=dashed](0,0)(0,4)
\end{pspicture} \qquad $\leftrightarrow$ \qquad
\begin{pspicture}(0,0)(10,10) \kdyckgrid{9}{9} \psline[linestyle=dashed](0,9)(9,0)
\psline(0,0)(0, 1)(0, 2)(0, 3)(0, 4)(0,
  5)(0, 6)(0, 7)(1, 7)(2, 7)(2, 8)(2, 9)(3, 9)(4, 9)(5, 9)(6, 9)(7,
  9)(8, 9)(9, 9)\psline(0,0)(0, 1)(0, 2)(1, 2)(1, 3)(2, 3)(3, 3)(3,
  4)(3, 5)(4, 5)(4, 6)(5, 6)(6, 6)(6, 7)(6, 8)(7, 8)(7, 9)(8, 9)(9,
  9)\pspolygon(3, 3)(2, 3)(2, 4)(2, 5)(2, 6)(3, 6)(3, 7)(4, 7)(5,
  7)(6, 7)(6, 6)(5, 6)(4, 6)(4, 5)(3, 5)(3, 4)(3, 3)\pspolygon(6,
  7)(5, 7)(5, 8)(5, 9)(6, 9)(7, 9)(7, 8)(6, 8)(6, 7)\pspolygon(5,
  7)(4, 7)(4, 8)(5, 8)\pspolygon(5, 8)(4, 8)(4, 9)(5, 9)\pspolygon(1,
  2)(0, 2)(0, 3)(0, 4)(1, 4)(2, 4)(2, 3)(1, 3)(1, 2)\pspolygon(2,
  4)(1, 4)(1, 5)(2, 5)\pspolygon(2, 5)(1, 5)(1, 6)(1, 7)(2, 7)(2,
  8)(3, 8)(4, 8)(4, 7)(3, 7)(3, 6)(2, 6)(2, 5)\pspolygon(4, 8)(3,
  8)(3, 9)(4, 9)\pspolygon(3, 8)(2, 8)(2, 9)(3, 9)\pspolygon(1, 4)(0,
  4)(0, 5)(1, 5)\pspolygon(1, 5)(0, 5)(0, 6)(1, 6)\pspolygon(1, 6)(0,
  6)(0, 7)(1, 7)
\psset{fillstyle=solid,fillcolor=lightgray}
\psframe(0,4)(1,5)
\psframe(0,5)(1,6)
\psframe(0,6)(1,7)
\psframe(2,8)(3,9)
\psframe(3,8)(4,9)
\psframe(4,8)(5,9)
\end{pspicture} 
  \caption{An illustration of the bijection between marked increasing
    labeling of a symmetric forest and symmetric Dyck tilings with
    fixed lower path. The newly added squares in each step are colored
  gray.}
  \label{fig:sym_bij}
\end{figure}
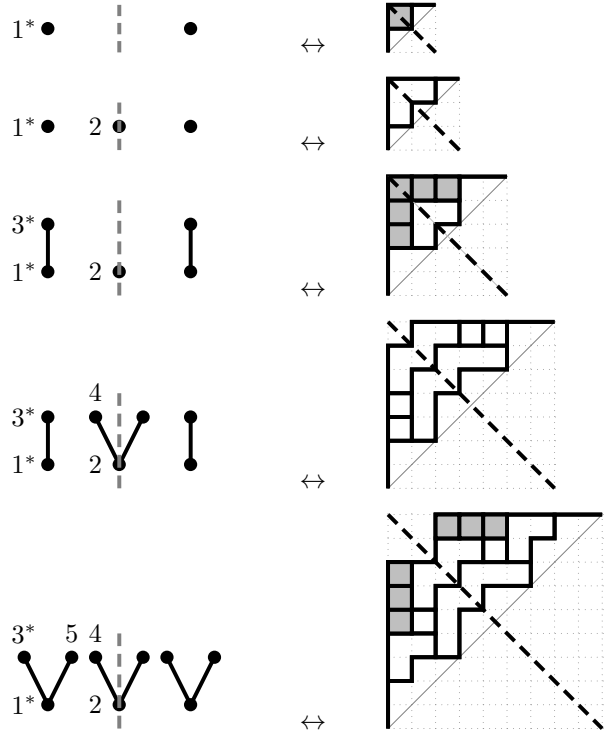

\begin{thm}
  Let $F$ be a symmetric plane forest of size $n$ and $\lambda$ the
  corresponding Dyck path. Then there is a
  bijection $\phi:\INC^*(F) \to \D_\sym(\lambda/*)$ such that if
  $\phi(L) = D$, then $\tiles_0(D) = |\MARK(L)|$ and 
\[
\art_+(D) + \art_0(D) = |\INV(L)| + \sum_{i\in\MARK(L)} (n+1-i).
\]
\end{thm}
\begin{proof}
  This is a generalization of a bijection in \cite{KMPW} constructed
  recursively. We ``spread'' and add ``broken strips'' to all up steps
  before the center line and to all down steps after the center
  line. If a vertex is marked, then we also add a square at the center
  line. See Figure~\ref{fig:sym_bij}.
\end{proof}

\begin{cor}
  Let $F$ be a symmetric plane forest of size $n$ with $k$ center
  vertices and $\lambda$ the corresponding Dyck path.  Then
\[
|\D_\sym(\lambda/*)| = 2^{n-k} \frac{n!}{\prod_{x\in F} h_x}.
\]
If $k=0$, we have
\[
\sum_{D\in\D_\sym(\lambda/*)} q^{\art_+(D) + \art_0(D)}t^{\tiles_0(D)} 
= (1+tq)(1+tq^2)\cdots(1+tq^n) \frac{[n]_q!}{\prod_{x\in F} [h_x]_q}.
\]
\end{cor}

\section{Symmetric Dyck tilings and symmetric Hermite histories}

For a symmetric Dyck path $\mu$ of length $2n$, let $\mu^+$ denote the
subpath consisting of the first $n$ steps. Note that each up step is
matched with a unique down step in a Dyck path. An up step of $\mu^+$
is called \emph{matched} if the corresponding down step in $\mu$ lies
in $\mu^+$ and \emph{unmatched} otherwise. 

A \emph{symmetric Hermite history} is a symmetric Dyck path $\mu$ with
a labeling of the up steps of $\mu^+$ in such a way that every matched
up step of height $h$ has label $i\in\{0,1,\dots, h-1\}$ and the
labels $a_1,a_2,\dots,a_k$ of the unmatched up steps form an
involutive sequence. Here, a sequence is called \emph{involutive} if
it can be obtained by the following inductive way.
\begin{itemize}
\item The empty sequence is defined to be involutive.
\item The sequence (0) is the only involutive sequence of length 1.
\item For $k\ge2$, an involutive sequence of length $k$ is either the
  sequence obtained from an involutive sequence of length $k-1$ by
  adding 0 at the end or the sequence obtained from an involutive
  sequence of length $k-2$ by adding an integer $r$ at the end and
  inserting a $0$ before the last $r$ integers, including the newly
  added integer, for some $1\le r\le k-1$.
\end{itemize}
From the definition it is clear that the number $v_k$ of involutive
sequences of length $k$ satisfies the recurrence $v_k = v_{k-1}+(k-1)
v_{k-2}$ with initial conditions $v_0=v_1=1$. Thus $v_k$ is equal to
the number of involutions in $\Sym_n$. 
We denote by $\HH_\sym(\mu)$ be the set of symmetric Hermite histories
on $\mu$. For $H\in \HH_\sym(\mu)$, let $\norm H$ be the sum of labels
in $H$ and $\pos(H)$ the number of positive labels on unmatched up
steps.

\begin{figure}
  \centering
\begin{pspicture}(0,0)(10,10)
\kdyckgrid{10}{10}
\psline(0,0)(0, 1)(0, 2)(0, 3)(0, 4)(0,
  5)(0, 6)(0, 7)(1, 7)(1, 8)(1, 9)(2, 9)(3, 9)(3, 10)(4, 10)(5, 10)(6,
  10)(7, 10)(8, 10)(9, 10)(10, 10)\psline[linestyle=dashed](0,10)(10,0)
\rput[r](-.3,.5){\bf \blue 0}
\rput[r](-.3,1.5){\bf \blue 0}
\rput[r](-.3,2.5){\bf \blue 2}
\rput[r](-.3,3.5){\bf \blue 0}
\rput[r](-.3,4.5){\bf \blue 2}
\rput[r](-.3,5.5){\bf \blue 4}
\rput[r](-.3,6.5){1}
\rput[r](.7,7.5){\bf \blue 1}
\rput[r](.7,8.5){\bf \blue 0}
\end{pspicture}
\caption{An example of symmetric Hermite history. The labels of the
  unmatched up steps are colored blue. The sequence
  $(0,0,2,0,2,4,1,0)$ of the labels of unmatched up steps is an
  involutive sequence.}
\label{fig:sym_hh}
\end{figure}
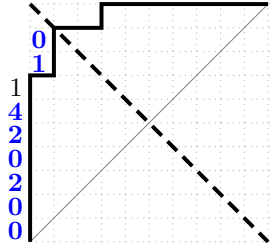

\begin{prop}
  There is a bijection $\psi:\HH_\sym(\mu)\to \D_\sym(*/\mu)$ such
  that if $\psi(H)=D$, then $\norm{H}=\tiles_+(D)+\tiles_0(D)$ and
  $\pos(H)=\tiles_0(D)$. Thus,
\[
\sum_{D\in\D_\sym(*/\mu)} q^{\tiles_+(D)+\tiles_0(D)} t^{\tiles_0(D)}
=\sum_{H\in\HH_\sym(\mu)} q^{\norm H} t^{\pos(H)}.
\]
\end{prop}
\begin{proof}
  Given a symmetric Dyck tiling $D$, considering $D$ as a normal Dyck
  tiling, we can obtain the Hermite history corresponding to $D$. By
  taking only the labels of up steps before the center line, we get a
  symmetric Hermite history. One can check that this gives a desired bijection.
\end{proof}

\begin{cor}\label{cor:sym}
  Let $\mu$ be a symmetric Dyck path such that $\mu^+$ has $k$
  unmatched up steps. Then
\[
\sum_{D\in\D_\sym(*/\mu)} q^{\tiles_+(D)+\tiles_0(D)} t^{\tiles_0(D)}
= f_k(q,t) \prod_{u\in\UP(\mu^+)} [\HT(u)]_q,
\]
where $\UP(\mu^+)$ is the set of matched up steps in $\mu^+$ and
$f_k(q,t)$ is defined by $f_0(q,t)=f_1(q,t)=1$ and
$f_k(q,t)=f_{k-1}(q,t)+tq[k-1]_q f_{k-2}(q,t)$ for $k\ge2$.
\end{cor}

A \emph{symmetric matching} is a matching on $[\pm n]$ such that if
$\{i,j\}$ is an arc, then $\{-i,-j\}$ is also an arc.  We denote by
$\M_\sym(n)$ the set of symmetric matchings on $[\pm n]$.  Note that
symmetric matchings are in bijection with fixed-point-free involutions
in $B_n$.

Let $M\in\M_\sym(n)$. A \emph{symmetric crossing} of $M$ is a pair of
arcs $\{a,b\}$ and $\{c,d\}$ satisfying $a<c<b<d$ and $b,d>0$. A
symmetric crossing $(\{a,b\},\{c,d\})$ is called \emph{self-symmetric}
if $\{c,d\}=\{-a,-b\}$. We denote by $\cro(M)$ and $\sscr(M)$ the
number of symmetric crossings and self-symmetric crossings of $M$,
respectively.

For a symmetric Dyck path $\mu$ of length $2n$, let
$\M_\sym(\mu)$ denote the set of symmetric matchings $M$ on $[\pm
n]$ such that the $i$th smallest vertex of $M$ is a left vertex of an
arc if and only if the $i$th step of $\mu$ is an up step.

\begin{prop}
We have
\[
\sum_{H\in\HH_\sym(\mu)} q^{\norm H} t^{\pos(H)}
=\sum_{M\in\M_\sym(\mu)} q^{\cro(M)} t^{\sscr(M)}.
\]
\end{prop}

For $D\in\D_\sym(*/\mu)$, let $\HT(D)$ denote the number of unmatched
up steps in $\mu^+$.  Using the result in \cite{Cigler2011} and
\cite{JV_rook} on a generating function for partial matchings we
obtain the following formula.

\begin{prop}
We have
\begin{multline*}
\sum_{D\in\D_\sym(n)} q^{\tiles_+(D)+\tiles_0(D)} t^{\tiles_0(D)} s^{\HT(D)}\\
= \sum_{m=0}^n \frac{s^m f_m(q,t)}{(1-q)^{(n-m)/2}}
\sum_{k\ge0} \left( \binom{n}{\frac{n-k}2} - \binom{n}{\frac{n-k}2-1}\right)
(-1)^{(k-m)/2} q^{\binom{(k-m)/2+1}2} 
\qbinom{\frac{k+m}2}{\frac{k-m}2},
\end{multline*}
where $f_m(q,t)$ is defined in Corollary~\ref{cor:sym}.
\end{prop}

If $t=s=0$ in the above proposition, then we get the generating
function for the usual Dyck tilings according to the number of tiles.

\end{document}